\documentclass[12pt]{article}
\usepackage{url,mathtools,bm}
\usepackage{enumerate,amsfonts, graphicx, amsmath}
\usepackage{graphicx}
\usepackage{epsfig}
\usepackage{epstopdf}
\usepackage{multirow,color}
\newcommand*{\myov}[1]{\overbracket[0.65pt][-1pt]{#1}}

\newcommand\obx{\mbox{${\myov{\bm{x}}}$}}
\newcommand\oobx{\mbox{${\myov{\myov{\bm{x}}}}$}}
\newcommand\oobw{\mbox{${\myov{\myov{{\bm{w}}}}}$}}
\newcommand\argmin{\mathop{\rm argmin}}

\newcommand\Span{\mathop{\rm span}}
\title{A single potential governing convergence of conjugate gradient, accelerated
gradient and geometric descent\thanks{Supported
in part by a grant from the U.~S.~Air Force Office of Scientific Research and in part
by a Discovery Grant from the Natural Sciences and Engineering Research Council
(NSERC) of Canada.}}
\author{Sahar Karimi\thanks{Department of Combinatorics \& Optimization,
University of Waterloo, 200 University Ave.~W., Waterloo, ON, N2L 3G1,
Canada, {\tt sahar.karimi@gmail.com}.} \and 
Stephen Vavasis\thanks{Department of Combinatorics \& Optimization,
University of Waterloo, 200 University Ave.~W., Waterloo, ON, N2L 3G1,
Canada, {\tt vavasis@uwaterloo.ca.}}}

\renewcommand\b{\bm{b}}
\renewcommand\c{\bm{c}}

\newcommand\g{\bm{g}}
\newcommand\p{\bm{p}}

\renewcommand\r{\bm{r}}
\newcommand\R{\mathbb{R}}

\renewcommand\v{\bm{v}}
\newcommand\w{\bm{w}}
\newcommand\x{\bm{x}}
\newcommand\y{\bm{y}}
\newcommand\z{\bm{z}}
\newcommand\eps{\epsilon}

\newcommand\bz{\bm{0}}
\newtheorem{lemma}{Lemma}
\newtheorem{theorem}{Theorem}
\newcommand{\eref}[1]{$(\ref{#1})$}
\newcommand{\aff}{\mathop{\rm aff}}
\newenvironment{proof}{\noindent {\bf Proof. }}{\hfill $\square$}
\begin{document}
\maketitle
\begin{abstract}
Nesterov's accelerated gradient (AG)
method for minimizing a smooth strongly convex function $f$ is
known to reduce $f(\x_k)-f(\x^*)$ by a factor 
of $\eps\in(0,1)$ after $k=O(\sqrt{L/\ell}\log(1/\eps))$ iterations, where
$\ell,L$ are the two parameters of smooth strong convexity. Furthermore,
it is known that this is the best possible complexity in the function-gradient oracle
model of computation.  Modulo a line search, the geometric descent (GD)
method of Bubeck, Lee and Singh has the same bound for this class of functions.
The method of linear conjugate gradients (CG) 
also satisfies the same
complexity bound in the special case of strongly convex quadratic functions,
but in this special case it can be faster than the AG and GD methods.

Despite similarities in the algorithms and their
asymptotic convergence rates, the conventional analysis of the
running time of CG is mostly disjoint
from that of AG and GD.  The analyses of 
the AG and GD methods are also rather distinct.

Our main result is analyses of the three methods that share several
common threads: all three analyses show
a relationship to a certain ``idealized algorithm'', all three
establish the convergence rate
through the use of the Bubeck-Lee-Singh geometric lemma, and all three
have the same potential
that is computable at run-time and exhibits decrease
by a factor of $1-\sqrt{\ell/L}$ or better
per iteration.

One application of these analyses is that they open the possibility of hybrid
or intermediate algorithms.  One such algorithm is proposed herein
and is shown to perform well in  computational tests.
\end{abstract}


\section{First-order methods for strongly convex functions}
Three methods for minimizing smooth, strongly
convex functions are considered in this work, conjugate gradient,
accelerated gradient, and geometric descent. 
CG is the oldest and perhaps best known of the methods.  It
was introduced by Hestenes and Stiefel \cite{hestenesstiefel}
for 
minimizing strongly convex quadratic functions
of the form $f(\x)=\x^TA\x/2-\b^T\x$, where $A$ is
a symmetric positive definite matrix.

There is a significant body of work on
gradient methods for more general
smooth, strongly convex functions.  We say that a 
differentiable convex function $f:\R^n\rightarrow\R$ is
{\em smooth, strongly convex} 
\cite{hiriarturrutylemarechal}
if there exist two scalars $L\ge\ell>0$ such that
for all $\x,\y\in\R^n$,
\begin{equation}
\ell\Vert\x-\y\Vert^2/2\le f(\y)-f(\x)-\nabla f(\x)^T(\y-\x)\le L\Vert\x-\y\Vert^2/2.
\label{eq:strconvdef}
\end{equation}
This is equivalent to assuming convexity and lower and upper
Lipschitz constants on the gradient:
$$\ell\Vert\x-\y\Vert \le \Vert\nabla f(\x)-\nabla f(\y)\Vert \le
L\Vert\x-\y\Vert.$$

Nemirovsky and Yudin \cite{NemYud83}
proposed
a method for minimizing smooth strongly convex
functions requiring
$k=O(\sqrt{L/\ell}\log(1/\eps))$ iterations to produce an iterate 
$\x_k$ such that $f(\x_k)-f(\x^*)\le\eps(f(\x_0)-f(\x^*))$, where
$\x^*$ is the optimizer (necessarily unique under the assumptions made).  
A drawback of their method is that it requires a two-dimensional
optimization on each iteration that can
be cumbersome to implement (to the best of our knowledge,
the algorithm was not ever widely adopted).
Nesterov \cite{Nesterov:k2} proposed
another method, nowadays
known as the ``accelerated gradient'' (AG) method, 
which achieves the same optimal complexity that requires a single
function and gradient evaluation on each iteration.

In the special case of strongly convex quadratic functions, 
the parameters $\ell$ and $L$ appearing in \eref{eq:strconvdef}
correspond to $\lambda_{\min}(A)$ and $\lambda_{\max}(A)$, the extremal
eigenvalues of $A$.  
The conjugate
gradient method has already been known to satisfy the asymptotic
iteration bound $k=O(\sqrt{L/\ell}\log(1/\eps))$ since the work
of Daniel (1967) described below.

Although the two methods satisfy the same asymptotic bound,
the analyses of the two methods are completely
different.  In the case of AG, there are
two analyses by Nesterov \cite{Nesterov:k2, Nesterov:book}.
In our own previous work \cite{KarimiVavasis2016}, we provided a third
analysis based on another potential.

In the case of linear
conjugate gradient, we are aware of no direct analysis of the algorithm
prior to our own previous work \cite{KarimiVavasis2016}.
By ``direct,'' we mean an analysis of $f(\x_{k})-f(\x^*)$ using the
recurrence inherent in CG.  Instead, the standard
analysis introduced by Daniel, whose theorem is
stated precisely below, proves that another iterative method, for
example Chebyshev iteration \cite{GVL} or the heavy-ball iteration
\cite{Polyak,bertsekas} achieves reduction of 
$\left(1-O(\sqrt{\ell/L})\right)$ per iteration.  Then one appeals
to the optimality of the CG iterate in the Krylov space generated
by all of these methods to claim that the CG iterate must be at
least as good as the others.

Recently, Bubeck, Lee and Singh \cite{bubeck} proposed the
geometric descent (GD) algorithm for
analysis of a variant of accelerated
gradient \cite{bubeck}, called ``geometric descent'' (GD).
As presented by the authors, the algorithm requires an exact
line-search on each iteration, although it is possible that similar theoretical
guarantees could be established for an approximate
line search.  Under the assumption
that the line-search requires a constant number of function and 
gradient evaluations, then GD also requires
$k=O(\sqrt{L/\ell}\log(1/\eps))$ iterations.

We propose analyses of these three algorithms that share several 
common features.  
First, all three algorithms 
can be analyzed using the geometric lemma
of Bubeck, Lee and Singh, which is presented in Section~\ref{sec:BLSlemma}.
Second, all three are related to an ``idealized'' unimplementable
algorithm which is described and analyzed in Section~\ref{sec:idealized}.  
Finally, the convergence
behavior for all three of them is governed by a potential $\tilde\sigma_k$,
which has the following three properties:
\begin{enumerate}
\item
There exists an auxiliary sequence of vectors $\y_0,\y_1,\ldots$ such that
$$\tilde\sigma_k^2 \ge \Vert\y_k-\x^*\Vert^2 +\frac{2(f(\x_k)-f(\x^*))}{\ell},$$
for $k=0,1,2\ldots,$
\item
$\displaystyle
\tilde\sigma_{k+1}^2 \le \left(1-\sqrt{\frac{\ell}{L}}\right)\tilde\sigma_k^2$, and
\item
$\tilde\sigma_k$ is computable on each iteration (assuming prior knowledge of
  $\ell,L$) in $O(n)$ operations.  Our definition of ``computable'' is
  explained in more detail in Section~\ref{sec:computable}.
\end{enumerate}
These results are established for the GD algorithm in Section~\ref{sec:GDanalysis},
for the CG algorithm in
Section~\ref{sec:CGanalysis2},
and for the 
AG algorithm in Section~\ref{sec:AGanalysis}.

The relationship between IA and the three algorithms is explained in detail
as follows.
Section~\ref{sec:CGanalysis} shows that CG exactly implements IA
for quadratic objective functions even though IA is in general unimplementable.
On the other hand, GD (analyzed in Section~\ref{sec:GDanalysis2}) and AG
(analyzed in Section~\ref{sec:AGanalysis2})
both simulate IA in the sense
that they produce optimal iterates given  partial information
about the objective in the current iterate.

Because the three algorithms each compute a scalar $\tilde \sigma_k$ satisfying
the above properties, it becomes straightforward to create hybrids.  In other
words, the above analysis treats all three algorithms as essentially 1-step
processes as opposed to long inductive chains.  In Section~\ref{sec:hybrid} we
propose a hybrid CG algorithm that performs well in computational tests, which
are described in Section~\ref{sec:comp}.
The reason from making a hybrid CG algorithm is that the performance
of linear conjugate gradient on specific instances can be much better than the
worst-case bound given by Daniel's theorem; the performance on specific
instances is highly governed by the eigenvalues of $A$.  Therefore, using
a conjugate-gradient-like algorithm for a nonlinear problem may also perform
better than the $(1-\sqrt{\ell/L})^k$ worst-case convergence bound.  This is
also the motivation for traditional nonlinear conjugate gradient, as we
discuss below.

We conclude this introductory section with a few remarks about our
previous related manuscript \cite{KarimiVavasis2016}.  In that work, we established
that a potential defined by
$$\Psi_k=\Vert\y_k-\x^*\Vert^2 +\frac{2(f(\x_k)-f(\x^*))}{\ell}$$
decreases by a factor $(1-\sqrt{\ell/L})$ per iteration for both CG
and AG.  This $\Psi_k$ is not computable since $\x^*$ is not known
{\em a priori}, and therefore our previous result does not have any
immediate algorithmic application.  Our notion of ``computability''
is defined in more detail in Section~\ref{sec:computable}.
The potential $\tilde\sigma_k$ developed
herein is computable on every step and therefore may be used to guide
a hybrid algorithm.  In addition, the current work also applies to
the GD method, which was not addressed in our previous manuscript.

\section{Notation}
Define $B(\x,r)=\{\y\in\R^n:\Vert \x-\y\Vert\le r\}$, i.e., the
closed ball centered at $\x\in\R^n$ of radius $r$.

An {\em affine set}
is a set of the form $\mathcal{M}=\{\c+\w:\w\in\mathcal{W}\}$ where $\c\in\R^n$ is fixed
and $\mathcal{W}\subset \R^n$ is a linear subspace.
We write this as $\mathcal{M}=\c+\mathcal{W}$, a special case of a Minkowski sum.
Another notation for an affine set is
$\aff\{\x_1,\ldots,\x_k\}$, which is defined as
$\{\alpha_1\x_1+\cdots+\alpha_k\x_k:\alpha_1+\cdots+\alpha_k=1\}$.
If $\w_1,\ldots,\w_k$ span $\mathcal{W}$, then
it is clear that 
$\c+\mathcal{W}=\aff\{\c,\c+\w_1,\ldots,\c+\w_k\}$.

Suppose $\mathcal{U}$ is an affine subset of $\R^n$.
The set ${\bf T}\mathcal{U}=\{\x-\y:\x,\y\in\mathcal{U}\}$ 
is called the {\em tangent space}
of $\mathcal{U}$ and is a linear subspace.  If $\mathcal{U}$ is presented as
$\mathcal{U}=\c+\mathcal{W}$, where $\mathcal{W}$ is a linear subspace, then
it follows that ${\bf T}\mathcal{U}=\mathcal{W}$.

\section{Preliminary lemmas}
\label{sec:BLSlemma}

We start with a special case of
a lemma from Drusvyatskiy et al.\ \cite{Drusvyatskiy}, which
is an extension of work by Bubeck et al.\ \cite{bubeck}:

\begin{lemma}
Suppose $\x,\y\in\R^n$.  Let
$\delta,\rho,\sigma$ be three nonnegative scalars
such that $\delta\le \Vert\x-\y\Vert$.
Suppose $\lambda\in[0,1]$ and 
\begin{equation}
\z=(1-\lambda)\x+\lambda\y.
\label{eq:lemmaz}
\end{equation}
Then
$$B(\x,\rho)\cap B(\y,\sigma)\subset B(\z,\xi),$$
where
\begin{equation}
\xi = \sqrt{(1-\lambda)\rho^2+\lambda\sigma^2-\lambda(1-\lambda)\delta^2},
\label{eq:xidef}
\end{equation}
The
argument of the square-root in \eref{eq:xidef} is guaranteed to be nonnegative
whenever $B(\x,\rho)\cap B(\y,\sigma)\ne \emptyset$,
or equivalently, whenever $\rho+\sigma\ge \Vert\x-\y\Vert$.
\label{lem:ballintersect1}
\end{lemma}

\begin{proof}
We prove the second claim first.
The quantity appearing in the square root of \eref{eq:xidef} is
nonnegative as the following inequalities show:
\begin{align*}
(1-\lambda)\rho^2 + \lambda\sigma^2 - (1-\lambda)\lambda\delta^2
&=(1-\lambda)\lambda(\rho^2 + \sigma^2 - \delta^2) +
(1-\lambda)^2\rho^2+\lambda^2\sigma^2 \\
&\ge (1-\lambda)\lambda(\rho^2 + \sigma^2 - \delta^2) +
2(1-\lambda)\lambda\rho\sigma \\
&= (1-\lambda)\lambda((\rho + \sigma)^2 - \delta^2) \\
&\ge 0,
\end{align*}
where the last line uses the assumptions $\rho+\sigma\ge\Vert\x-\y\Vert\ge \delta$.

Now for the first part of the lemma,
the proof that $B(\x,\rho)\cap B(\y,\sigma)\subset B(\z,\xi)$
for $\xi$ given by \eref{eq:xidef} follows from more general
analysis in Drusvyatskiy et al.~\cite{Drusvyatskiy}.  
Assume that $\p\in B(\x,\rho)\cap B(\y,\sigma)$ so
\begin{align}
(\p-\x)^T(\p-\x) -\rho^2 &\le 0, \label{eq:px1} \\
(\p-\y)^T(\p-\y) -\sigma^2 &\le 0. \label{eq:py1}
\end{align}
For $\lambda\in[0,1]$, add $(1-\lambda)$ times \eref{eq:px1} to
$\lambda$ times \eref{eq:py1} and rearrange to obtain a new inequality satisfied
by $\p$:
$$(\p-\z)^T(\p-\z)+(1-\lambda)\lambda\Vert\x-\y\Vert^2 -(1-\lambda)\rho^2-
\lambda\sigma^2\le 0,$$
i.e.
$$\Vert\p-\z\Vert\le \left((1-\lambda)\rho^2+\lambda\sigma^2-(1-\lambda)\lambda\Vert\x-\y\Vert^2\right)^{1/2},$$
where $\z$ is defined by \eref{eq:lemmaz}.  By substituting the definition
$\delta\le\Vert\x-\y\Vert$, we observe that $\p\in B(\z,\xi)$, where $\xi$ is defined
by \eref{eq:xidef}.
\end{proof}

This leads to the following, which is
a more precise statement of the geometric lemma from
Bubeck et al.\ \cite{bubeck}:

\begin{lemma}
Let $\x,\y,\rho,\sigma,\delta$ be as in the preceding lemma.
Under the assumption 
$\rho+\sigma\ge \delta$ and the
additional assumption
$\delta \ge \sqrt{|\rho^2-\sigma^2|}$, 
\eref{eq:xidef} is minimized 
over possible choices of $\lambda\in[0,1]$ by:
\begin{equation}
\lambda^* = \frac{\delta^2+\rho^2-\sigma^2}{2\delta^2},
\label{eq:lambdastar}
\end{equation}
yielding
\begin{equation}
\z^*=(1-\lambda^*)\x+\lambda^*\y,
\label{eq:zdef}
\end{equation}
in which case the minimum value of \eref{eq:xidef} is,
\begin{equation}
\xi^*=\frac{1}{2}\sqrt{2\rho^2+2\sigma^2-\delta^2-
\frac{(\rho^2-\sigma^2)^2}{\delta^2}}.
\label{eq:xistardef}
\end{equation}
\label{lem:ballintersect2}
\end{lemma}

\begin{proof}
First, note that
$|\rho^2-\sigma^2|/\delta^2\le 1$ by the assumption made, thus
ensuring that $\lambda^*\in[0,1]$.  
Therefore, it follows from the preceding lemma that
the quantity appearing in the square root of \eref{eq:xidef} is nonnegative.
The previous lemma establishes that for any $\p\in B(\x,\rho)\cap B(\y,\sigma)$,
and for an arbitrary $\lambda\in[0,1]$,
$$
\Vert\p-\z\Vert^2 \le (1-\lambda)\rho^2 + \lambda\sigma^2 - \lambda(1-\lambda)\delta^2.$$
We observe that the right-hand side is a convex quadratic in $\lambda$ and
hence is minimized when the derivative with respect to $\lambda$ is zero, and one
checks that this value is precisely \eref{eq:lambdastar}.  Substituting 
$\lambda=\lambda^*$ into \eref{eq:xidef} yields \eref{eq:xistardef}.
\end{proof}

\section{Idealized algorithm}
\label{sec:idealized}

We consider the following idealized algorithm for minimizing $f(\x)$, where
$f:\R^n\rightarrow\R$ is smooth, strongly convex.  As in the introduction,
let $\ell,L$ denote the two parameters of strong convexity.

\begin{align}
& \mbox{\bf Idealized Algorithm (IA)} \notag\\
& \x_0:=\mbox{arbitrary} \notag \\
& \mathcal{M}_1 := \x_0+\Span\{\nabla f(\x_0)\} \notag \\
&\mbox{for }k:=1,2,\ldots \notag  \\
&\hphantom{\mbox{for }} \x_{k}:=\argmin\{f(\x):\x\in \mathcal{M}_{k}\}\label{eq:ia2.xupd} \\
& \hphantom{\mbox{for }} \y_k := \argmin\{\Vert\y-\x^*\Vert:\y\in\mathcal{M}_{k}\}
\label{eq:ia2.yupd}\\
&\hphantom{\mbox{for }} \mathcal{M}_{k+1}:=\x_{k}+\Span\{\y_k-\x_k,\nabla f(\x_{k})\} \label{eq:ia2.mupd}\\
& \mbox{end} \notag
\end{align}

This algorithm is called ``idealized'' because it is not implementable 
in the general case;
it requires prior knowledge of $\x^*$ in 
\eref{eq:ia2.yupd}.  Nonetheless, we will argue
that CG, accelerated gradient, and geometric gradient are related
to the idealized algorithm in different ways. 

Notice that $\mathcal{M}_k$ is an affine set that is two-dimensional
on most iterations.
Alternate notation for this set, also used herein, is
$\mathcal{M}_{k}=\aff\{\x_{k-1},\y_{k-1},\x_{k-1}-\nabla f(\x_{k-1})\}$.
Note also that by \eref{eq:ia2.xupd} and \eref{eq:ia2.yupd},
$\x_k,\y_k\in\mathcal{M}_k$, and by \eref{eq:ia2.mupd},
$\x_k,\y_k\in\mathcal{M}_{k+1}$, and therefore $\mathcal{M}_k$,
$\mathcal{M}_{k+1}$ have a common 1-dimensional affine subspace.

We start with the main theorem about
IA.  For iteration $k$, define
a potential $\Psi_k$ as follows:
\begin{equation}
\Psi_k= \Vert \y_{k}-\x^*\Vert^2+\frac{2(f(\x_k)-f(\x^*))}{\ell}.
\label{eq:psidef}
\end{equation}

\begin{theorem}
For Algorithm IA, for each $k=1,2,\ldots$, 
$$\Psi_{k+1}\le \left(1-\sqrt{\frac{\ell}{L}}\right)\Psi_k.$$
\label{thm:iacvg}
\end{theorem}

\begin{proof}
The proof follows closely from the analysis in \cite{bubeck}.  Define
\begin{align*}
\obx_k&=\x_k-\nabla f(\x_k)/L, \\
\oobx_k&=\x_k-\nabla f(\x_k)/\ell.
\end{align*}
The point $\obx_k$
satisfies $f(\obx_k)\le f(\x_k)-\Vert\nabla f(\x_k)\Vert^2/(2L)$.
Observe that
$\obx_k\in\mathcal{M}_{k+1}$,
so $\obx_k$ is a candidate
for the optimizer in \eref{eq:ia2.xupd} on iteration $k+1$, and hence
\begin{equation}
f(\x_{k+1})\le f(\x_k)-\Vert\nabla f(\x_k)\Vert^2/(2L),
\label{eq:fdesc0}
\end{equation}
which is equivalent to
\begin{equation}
\frac{2(f(\x_{k+1})-f(\x^*))}{\ell}\le \frac{2(f(\x_k)-f(\x^*))}{\ell}-
\frac{\Vert\nabla f(\x_k)\Vert^2}{L\ell}.
\label{eq:fdesc}
\end{equation}

Next, observe that a rearrangement of the definition of strong convexity yields:
\begin{equation}
\frac{-2\nabla f(\x_k)^T(\x_k-\x^*)}{\ell}
+\Vert\x_k-\x^*\Vert^2
\le \frac{-2(f(\x_k)-f(\x^*))}{\ell}.
\label{eq:strcvx1}
\end{equation}
We use this result in the following:
\begin{align}
\Vert\oobx_k-\x^*\Vert^2 & = \Vert\oobx_k-\x_k+\x_k-\x^*\Vert^2 \notag \\
&= \Vert\oobx_k-\x_k\Vert^2 + 2(\oobx_k-\x_k)^T(\x_k-\x^*) + \Vert\x_k-\x^*\Vert^2 \notag\\
&=\frac{\Vert\nabla f(\x_k)\Vert^2}{\ell^2} -
\frac{2\nabla f(\x_k)^T(\x_k-\x^*)}{\ell} +\Vert\x_k-\x^*\Vert^2 \notag \\
&\le 
\frac{\Vert\nabla f(\x_k)\Vert^2}{\ell^2} 
-\frac{2(f(\x_k)-f(\x^*))}{\ell}\mbox{ (by \eref{eq:strcvx1})}\label{eq:rdef0} \\
&\equiv \rho_k^2,\label{eq:rdef}
\end{align}
where we introduced $\rho_k$ for the square root of the quantity in \eref{eq:rdef0}.
Thus, $\x^*\in B(\oobx_k,\rho_k)$.

Next, define
\begin{equation}
\sigma_k=\Vert\y_k-\x^*\Vert,
\label{eq:sdef}
\end{equation}
so that $\x^*\in B(\y_k,\sigma_k)$.

By the minimality
property of $\x_k$, we know that $\nabla f(\x_k)$ is orthogonal to 
${\bf T}\mathcal{M}_{k}$,
which contains $\x_k-\y_k$, i.e.,
\begin{equation}
\nabla f(\x_k)^T(\y_k-\x_k)=0.
\label{eq:gradorth}
\end{equation}
Thus,
\begin{align}
\Vert\y_k-\oobx_k\Vert 
&= \Vert(\y_k-\x_k)+(\x_k-\oobx_k)\Vert \notag \\
&= \Vert(\y_k-\x_k)+\nabla f(\x_k)/\ell\Vert \notag \\
&= \sqrt{\Vert\y_k-\x_k\Vert^2+\Vert\nabla f(\x_k)/\ell\Vert^2} \mbox{ (by Pythagoras's theorem)}\notag \\
&\ge  \Vert\nabla f(\x_k)\Vert/\ell, \label{eq:ykoobx}
\end{align}
so define
\begin{equation}
\delta_k =  \Vert\nabla f(\x_k)\Vert/\ell,
\label{eq:ddef}
\end{equation}
to conclude that $\Vert\y_k-\oobx_k\Vert\ge \delta_k$.  
We have defined $\delta_k,\rho_k,\sigma_k$ as in
Lemma~\ref{lem:ballintersect2}.
We need to confirm the inequality
$\rho_k+\sigma_k\ge \delta_k$:
\begin{align*}
\rho_k+\sigma_k&\ge \Vert\oobx_k-\x^*\Vert + \Vert \y_k-\x^*\Vert \\
& \ge  \Vert\oobx_k-\y_k\Vert \\
&\ge \delta_k.
\end{align*}
The other inequality is derived as follows. First,
$\rho_k\le \delta_k$ since $\delta_k^2$ is the first term in \eref{eq:rdef0}.
Also, $\sigma_k\le \delta_k$ since
\begin{align*}
\sigma_k^2 &= \Vert \y_k-\x^*\Vert^2 \\
&\le \Vert \x_k-\x^*\Vert^2 && \mbox{(by the optimality of $\y_k$)} \\
&\le \frac{2(f(\x_k)-f(\x^*))}{\ell} && \mbox{(by strong convexity)} \\
&\le \delta_k^2 && \mbox{(since \eref{eq:rdef0} is nonnegative)}.
\end{align*}
Thus, $\delta_k\ge\max(\rho_k,\sigma_k)$ so $\delta_k^2\ge |\rho_k^2-\sigma_k^2|$.

Therefore, we can conclude from  Lemma~\ref{lem:ballintersect2} that
there exists a $\z_k^*\in\aff\{\oobx_k,\y_k\}$ (and hence in 
$\mathcal{M}_{k+1}$)
such that 
\begin{equation}
\Vert\z_k^*-\x^*\Vert\le \xi_k^*,
\label{eq:zbound}
\end{equation}
where $\xi_k^*$ is defined by \eref{eq:xistardef} for $\rho_k,\sigma_k,\delta_k$ 
given by \eref{eq:rdef},
\eref{eq:sdef} and \eref{eq:ddef} respectively.  After some simplification
and cancellation of \eref{eq:xistardef}, one arrives at:
\begin{equation}
(\xi_k^*)^2=\Vert\y_k-\x^*\Vert^2 - \left(\frac{f(\x_k)-f(\x^*)+\Vert\y_k-\x^*\Vert^2\cdot\ell/2}
{\Vert\nabla f(\x_k)\Vert}\right)^2.
\label{eq:pkdef}
\end{equation}
Since $\y_{k+1}$ is the optimizer of \eref{eq:ia2.yupd}, 
$\y_{k+1}$ is
at least as close to $\x^*$ as $\z_k^*$, and hence, 
$$\Vert\y_{k+1}-\x^*\Vert^2\le 
\Vert\y_k-\x^*\Vert^2 - \left(\frac{f(\x_k)-f(\x^*)+\Vert\y_k-\x^*\Vert^2\cdot\ell/2}
{\Vert\nabla f(\x_k)\Vert}\right)^2.$$
Adding this inequality to \eref{eq:fdesc} yields:
\begin{align*}
\Psi_{k+1}&=\Vert\y_{k+1}-\x^*\Vert^2 + \frac{2(f(\x_{k+1})-f(\x^*))}{\ell} \\
&\le \Vert\y_k-\x^*\Vert^2 - \left(\frac{f(\x_k)-f(\x^*)+\Vert\y_k-\x^*\Vert^2\cdot\ell/2}
{\Vert\nabla f(\x_k)\Vert}\right)^2 + \frac{2(f(\x_k)-f(\x^*))}{\ell}-
\frac{\Vert\nabla f(\x_k)\Vert^2}{L\ell} \\
& \le
\Vert\y_k-\x^*\Vert^2 -\frac{2\left[f(\x_k)-f(\x^*)+\Vert\y_k-\x^*\Vert^2\cdot\ell/2\right]}{\sqrt{L\ell}}
+ \frac{2(f(\x_k)-f(\x^*))}{\ell}  \\
&=
\left[\Vert\y_k-\x^*\Vert^2+\frac{2(f(\x_k)-f(\x^*))}{\ell}\right]
\cdot\left(1-\sqrt{\frac{\ell}{L}}\right) \\
&=
\Psi_k
\cdot\left(1-\sqrt{\frac{\ell}{L}}\right).
\end{align*}
The third line was obtained by applying the inequality
 $a^2+b^2\ge 2ab$ to the second and fourth terms of the second line.
\end{proof}

\section{Analysis of the geometric descent algorithm}
\label{sec:GDanalysis}

In this section we present the
geometric descent (GD) algorithm due to
\cite{bubeck} and
an analysis of it. 
Our analysis varies slightly from the proof
due to \cite{bubeck}; in their proof, the
potential involves the term $2(f(\obx_k)-f(\x^*))/\ell$
rather than $2(f(\x_k)-f(\x^*))/\ell$. 
The reason for
the change is to unify the analysis with the other algorithms
considered in order for the NCG construction in Section~\ref{sec:hybrid}
to be applicable.

\begin{align}
& \mbox{\bf Geometric Descent} \notag\\
& \x_0:=\mbox{arbitrary} \notag \\
& \y_0:=\x_0 \notag\\
&\mbox{for }k:=1,2,\ldots \notag  \\
& \hphantom{\mbox{for }} {\obx}_{k-1} : =\x_{k-1}-\nabla f(\x_{k-1})/L \label{eq:gd.steep} \\
& \hphantom{\mbox{for }} {\oobx}_{k-1} : =\x_{k-1}-\nabla f(\x_{k-1})/\ell \notag \\
&\hphantom{\mbox{for }} \mbox{Determine $\lambda_k$ according to
\eref{eq:gd.lambda1} or \eref{eq:gd.lambda2} below.} \notag \\
&\hphantom{\mbox{for }} \y_k:=(1-\lambda_k)\oobx_{k-1} + \lambda_k\y_{k-1} \label{eq:gd.yupd} \\
&\hphantom{\mbox{for }} \x_k:=\argmin\{f(\x):\x\in \aff\{\obx_{k-1},\y_k\}\}\label{eq:gd.xupd} \\
& \mbox{end} \notag
\end{align}

Note: The operation in \eref{eq:gd.xupd} is a line search and
requires an inner iteration to find the optimal $\x$ in the specified line.

This algorithm (which is one of several variants
of GD presented by the authors) is derived in \cite{bubeck}.
It can be regarded as
extracting the essential properties
of $\x_k$ and $\y_k$ used in the proof of Theorem~\ref{thm:iacvg} to
obtain an implementable algorithm.
Indeed, the authors use that proof that we presented
in Section~\ref{sec:idealized} to analyze GD rather 
IA.  A more precise statement of the relationship between
GD and IA is provided in the next section.

Intuitively, the proof in the previous
section shows that $\x_k$ need not be the minimizer
in \eref{eq:ia2.xupd}; it suffices for $\x_k$ to satisfy the two
properties \eref{eq:fdesc0} and \eref{eq:gradorth}.  The Geometric
Descent algorithm satisfies these two properties with a ``dogleg'' step
in \eref{eq:gd.xupd}
that combines a gradient step with a step toward $\y_k$.
Property \eref{eq:fdesc0} is satisfied because $f(\x_k)\le f(\obx_{k-1})$,
and property \eref{eq:gradorth} is satisfied because of the minimality of
$\x_k$ with respect to $\aff\{\obx_{k-1},\y_k\}.$

The proof also shows that it suffices to take a $\y_{k+1}$ that
satisfies the inequality for $\z_k^*$ in 
\eref{eq:zbound} rather than solving \eref{eq:ia2.yupd}.

We now turn to the computation of $\lambda_k$ and the
associated issues with the radii $\rho_k,\sigma_k$.
Recall that $\rho_k^{\rm IA}$ from \eref{eq:rdef0}
and $\sigma_k^{\rm IA}$ from \eref{eq:sdef} both involve
$\x^*$ and hence are unimplementable.  The difficulty with
$\sigma_k$ is straightforward to resolve: define $\sigma_k$
to be an upper on $\Vert\y_k-\x^*\Vert$ rather than its exact value,
and ensure inductively that $\sigma_{k+1}$ is an upper
bound on $\Vert\y_{k+1}-\x^*\Vert$.

The difficulty with \eref{eq:rdef0} is resolved using 
offsets denoted by $\gamma_k$, 
a clever  device from \cite{bubeck}.  As in \eref{eq:rdef0} and \eref{eq:rdef},
\begin{align}
\Vert\oobx_k-\x^*\Vert^2 
&\le 
\frac{\Vert\nabla f(\x_k)\Vert^2}{\ell^2} 
-\frac{2(f(\x_k)-f(\x^*))}{\ell}\mbox{ (by \eref{eq:strcvx1})}\label{eq:gd.rdef0} \\
&\equiv \rho_k^2,\label{eq:gd.rdef} \\
& \equiv \tilde\rho_k^2 - \gamma_k \label{eq:gd.rdef2}
\end{align}
where
\begin{align}
\tilde\rho_k &= \frac{\Vert\nabla f(\x_k)\Vert}{\ell}, \label{eq:rpdef} \\
\gamma_k &= \frac{2(f(\x_k)-f(\x^*))}{\ell}. \label{eq:gammakdef}
\end{align}
Let  $\sigma_0,\sigma_1,\ldots,$ be a sequence of positive scalars 
such that $\sigma_k\ge \Vert\y_k-\x^*\Vert$ for all $k=0,1,\ldots$,
and suppose that
\begin{equation}
\tilde\sigma_k = (\sigma_k^2 + \gamma_k)^{1/2}.
\label{eq:sigmakprime}
\end{equation}
Thus, we have the relationships:
\begin{align*}
\tilde\sigma_k^2 &= \sigma_k^2 + \gamma_k, \\
\tilde\rho_k^2 &= \rho_k^2 + \gamma_k.
\end{align*}
Note that $\tilde\rho_k$ is easily computable on the $k$th iteration, while
$\tilde\sigma_k$ can be updated recursively.  The rationale of these definitions is
as follows. From \eref{eq:lambdastar}, one sees that
if $\sigma^2$ and $\rho^2$ are both incremented by the same constant
additive term $\gamma_k$, 
then $\lambda^*$ is unaffected.  Also, it follows from this observation
and from \eref{eq:xistardef} that if $\sigma^2$ and $\rho^2$ are both incremented
by $\gamma_k$, then $(\xi^*)^2$ is also incremented by $\gamma_k$.  Thus,
the GD algorithm works throughout with radii whose squares are incremented by $\gamma_k$.
This increment $\gamma_k$ changes from one iteration to the next and hence must
be adjusted at the start of each iteration (see \eref{eq:sigmadef} below).

In more detail, the sequence of computations is as follows.
We initialize the algorithm by:
\begin{equation}
  \sigma_0:=\sqrt{2}\Vert\nabla f(\x_0)\Vert/\ell. \label{eq:sigma0}
\end{equation}
This initialization is carried out prior to the main loop of GD.
The rationale for this formula is provided in the proof of
\eref{eq:sigma-ub} below.

Assuming inductively that $\tilde\sigma_{k-1}$ is already known, 
compute as follows:
\begin{align}
& \tilde\rho_{k-1} := \frac{\Vert\nabla f(\x_{k-1})\Vert}{\ell}
\label{eq:gd.rpdef} \\
&\mbox{if $\tilde\sigma_{k-1}^2\le 2\tilde\rho_{k-1}^2$} \label{eq:gd.ifstmt} \\
& \hphantom{\mbox{if }}\delta_{k-1} := \Vert\y_{k-1}-\oobx_{k-1}\Vert, \label{eq:gd.deltadef} \\
& \hphantom{\mbox{if }}\lambda_k := 
\frac{\delta_{k-1}^2+\tilde\rho_{k-1}^2-\tilde\sigma_{k-1}^2}{2\delta_{k-1}^2}, 
&&\mbox{(as in \eref{eq:lambdastar})} 
\label{eq:gd.lambda1}\\
& \hphantom{\mbox{if }}\tilde\xi^*_k :=
\frac{1}{2}\sqrt{2\tilde\rho_{k-1}^2+2\tilde\sigma_{k-1}^2-\delta_{k-1}^2-
\frac{(\tilde\rho_{k-1}^2-\tilde\sigma_{k-1}^2)^2}{\delta_{k-1}^2}},
&&\mbox{(as in \eref{eq:xistardef})} \label{eq:gd.xik} \\
&\hphantom{\mbox{if }} \tilde\sigma_k := \sqrt{(\tilde\xi^*_k)^2-\gamma_{k-1}+\gamma_{k}} \label{eq:sigmadef} \\
& \mbox{else} \\
& \hphantom{\mbox{if }}\lambda_k := 0 \label{eq:gd.lambda2}\\
& \hphantom{\mbox{if }}\tilde\sigma_k := \sqrt{\tilde\rho_{k-1}^2 - \gamma_{k-1}+\gamma_{k}}.
\label{eq:sigmadef2} 
\end{align}
Although computation of $\gamma_k$ alone requires prior knowledge of $\x^*$, 
the difference $\gamma_{k}-\gamma_{k-1}$ appearing
in \eref{eq:sigmadef} and \eref{eq:sigmadef2} does not, as is
evident from \eref{eq:gammakdef}.  This is the motivation for
using $\tilde\sigma_k$ and $\tilde\rho_k$ in the computation instead
of $\sigma_k,\rho_k$.
In the theorems below
it is confirmed that the square
roots in \eref{eq:sigmadef} and \eref{eq:sigmadef2} take nonnegative
arguments.

The convergence of the GD algorithm is proved via
two theorems, which are both variants of
theorems due to \cite{bubeck}.

Before stating and proving the two theorems, we establish two
inequalities.  
As noted earlier, \eref{eq:gradorth} holds for GD, and
hence so does \eref{eq:ykoobx}.
Then it follows from \eref{eq:ykoobx}
combined with \eref{eq:gd.rpdef}, \eref{eq:gd.deltadef} that
\begin{equation}
\tilde\rho_{k-1}= \frac{\Vert \nabla f(\x_{k-1})\Vert}{\ell} \le \delta_{k-1}.
\label{eq:gd.deltarhobd}
\end{equation}

Next, regarding $\gamma_{k}-\gamma_{k-1}$ appearing in
\eref{eq:sigmadef} and \eref{eq:sigmadef2}, observe
\begin{align}
\gamma_{k}-\gamma_{k-1} & = \frac{2(f(\x_{k})-f(\x_{k-1}))}{\ell} \notag\\
& \le \frac{2(f(\obx_{k-1})-f(\x_{k-1}))}{\ell}\notag\\
&\le -\frac{\Vert\nabla f(\x_{k-1})\Vert^2}{L\ell} \label{eq:gammadiff}
\end{align}
where the second line follows because $f(\x_{k})\le f(\obx_{k-1})$ by
\eref{eq:gd.xupd} while the third follows by \eref{eq:fdesc0}.

\begin{theorem}
For all $k=1,\ldots,$
\begin{equation}
\tilde\sigma_k^2 \ge \Vert\y_{k}-\x^*\Vert^2 + \frac{2(f(\x_k)-f(\x^*))}{\ell}.
\label{eq:sigma-ub}
\end{equation}
\end{theorem}

\begin{proof}
Note that the statement of the theorem may be equivalently written,
$$\tilde\sigma_k^2 \ge \Vert\y_{k}-\x^*\Vert^2 + \gamma_{k}.$$

The proof is by induction.  
The base case is that $\tilde\sigma_0^2\ge \Vert\y_0-\x^*\Vert^2+\gamma_0$.
Both terms of the right-hand side may be bounded by noting that strong convexity applied
to the two points $\x_0,\x^*$ and rearranged may be written:
$$-\frac{2}{\ell}\nabla f(\x_0)^T(\x^*-\x_0)\ge \Vert \x_0-\x^*\Vert^2 +\gamma_0.$$
Again by strong convexity, $\Vert\x_0-\x^*\Vert\le \Vert \nabla f(\x_0)\Vert/\ell$, so
we can apply this inequality and the Cauchy-Schwarz inequality on the left-hand
side to obtain
$$\frac{2\Vert\nabla f(\x_0)\Vert^2}{\ell^2}\ge  \Vert\x_0-\x^*\Vert^2 +\gamma_0.$$
Thus, the bound
$\tilde\sigma_0^2\ge\Vert\y_0-\x^*\Vert^2+\gamma_0$ is assured
by \eref{eq:sigma0}.

For the induction case, assume $k\ge 1$ and
the induction hypothesis
$$\tilde\sigma_{k-1}^2 \ge \Vert\y_{k-1}-\x^*\Vert^2 + \gamma_{k-1}.$$
There are two possibilities depending on the ``if''-statement 
\eref{eq:gd.ifstmt}.  First, suppose the condition
of \eref{eq:gd.ifstmt} holds, which may
be rewritten as
$\tilde\sigma_{k-1}^2-\tilde\rho_{k-1}^2 \le \tilde\rho_{k-1}^2$.
By \eref{eq:gd.deltarhobd}, this implies  $\delta_{k-1}^2\ge \tilde\sigma_{k-1}^2-\tilde\rho_{k-1}^2$,
and we already know from \eref{eq:gd.deltarhobd} that
$\delta_{k-1}^2\ge\tilde\rho_{k-1}^2\ge \tilde\rho_{k-1}^2-\tilde\sigma_{k-1}^2$.  
The conclusion from all these inequalities is
\begin{equation}
\delta_{k-1}^2\ge\tilde\rho_{k-1}^2\ge |\tilde\rho_{k-1}^2-\tilde\sigma_{k-1}^2|.
\label{eq:gd.chain1}
\end{equation}
Referring now to Lemma~\ref{lem:ballintersect2}, make the following identifications:
\begin{align*}
\x &= \oobx_{k-1},\\
\y &= \y_{k-1}, \\
\rho^2 & = \tilde\rho_{k-1}^2 - \gamma_{k-1}, \\
\sigma^2 &= \tilde\sigma_{k-1}^2 - \gamma_{k-1}, \\
\delta &= \delta_{k-1},
\end{align*}
in order to apply the lemma.  The condition $\rho+\sigma\ge \delta$
follows immediately since we already have assumed by induction
that $\x^*\in B(\y,\sigma)$ and it follows from \eref{eq:gd.rdef}
that $\x^*\in B(\x,\rho)$, thus implying that
$B(\x,\rho)\cap B(\y,\sigma)\ne\emptyset$.  The condition
$\delta^2\ge\rho^2-\sigma^2$ follows because $\delta_{k-1}\ge\tilde\rho_{k-1}$
as in \eref{eq:gd.chain1}.  The condition 
$\delta^2\ge \sigma^2-\rho^2$ follows because
$\delta_k^2\ge \tilde\sigma_{k-1}^2-\tilde\rho_{k-1}^2$ (as established in
\eref{eq:gd.chain1}).
Therefore, by the lemma,
if we define $\lambda^*$ 
$$\lambda^*=\frac{\delta^2+\rho^2-\sigma^2}{2\delta^2}
=\frac{\delta_{k-1}^2 +\tilde\rho_{k-1}^2 
-\tilde\sigma_{k-1}^2}{2\delta_{k-1}^2},$$
i.e., the formula for $\lambda_k$ in \eref{eq:gd.lambda1}
(notice that the two terms $\gamma_{k-1}$ cancel),
and we define $\y_k$ as in \eref{eq:gd.yupd}, 
then
\begin{equation}
\x^*\in B(\y_k,\xi)
\label{eq:gd.xi1}
\end{equation}
where
\begin{align*}
\xi^2 &= \frac{1}{4}\left(2\rho^2 + 2\sigma^2 - \delta^2 -\frac{1}{\delta^2}
(\rho^2-\sigma^2)^2\right) \\
& = 
\frac{1}{4}
\left(
2\tilde\rho_{k-1}^2 
+2\tilde\sigma_{k-1}^2 - 4\gamma_{k-1} - \delta_{k-1}^2 
-\frac{1}{\delta_{k-1}^2}
\left(
\tilde\rho_{k-1}^2 - \tilde\sigma_{k-1}^2\right)^2\right) \\
&=
(\tilde\xi_k^*)^2 - \gamma_{k-1} &&\mbox{(by \eref{eq:gd.xik})} \\
& =
\tilde\sigma_k^2 - \gamma_k, &&\mbox{(by \eref{eq:sigmadef})}.
\end{align*}
This establishes
the theorem in the first case.

If the condition in \eref{eq:gd.ifstmt} fails, then
$\lambda_k=0$ as in \eref{eq:gd.lambda2}, implying from
\eref{eq:gd.yupd} that $\y_k=\oobx_{k-1}$.
Then $\Vert\y_k-\x^*\Vert^2\le \tilde\rho_{k-1}^2-\gamma_{k-1}$
by 
\eref{eq:gd.rdef2}, implying by \eref{eq:sigmadef2}
that $\Vert\y_k-\x^*\Vert^2 \le \tilde\sigma_k^2-\gamma_k$,
thus establishing the theorem in the second case.
\end{proof}

\begin{theorem}
For each $k=1,2,\ldots,$
\begin{equation}
\tilde\sigma_k^2 \le \left(1-\sqrt{\frac{\ell}{L}}\right)\tilde\sigma_{k-1}^2.
\label{eq:sigma-decr}
\end{equation}
\end{theorem}

\begin{proof}
We again take two cases depending on whether
the condition in \eref{eq:gd.ifstmt} holds.
If it holds, then 
\eref{eq:gd.chain1} in the preceding proof
holds.  
Observe that the function $x\mapsto x+C/x$ for $C>0$ is unimodal on 
$(0,\infty)$  with a minimizer at $\sqrt{C}$, which means that if
all the other parameters are fixed, the maximizing choice for $\delta_{k-1}^2$
in \eref{eq:gd.xik} is $\delta_*^2=|\tilde\rho_{k-1}^2-\tilde\sigma_{k-1}^2|$.
Therefore, by unimodality combined with the ordering
$\delta_{k-1}\ge\tilde\rho_{k-1}\ge\delta_*$ 
(which is \eref{eq:gd.chain1}),
the right-hand side of \eref{eq:gd.xik} can only
increase if we replace $\delta_{k-1}$ by $\tilde\rho_{k-1}$, thus obtaining,
\begin{align*}
(\tilde\xi^*_k)^2
&=
\frac{1}{4}\left(2\tilde\rho_{k-1}^2+2\tilde\sigma_{k-1}^2-\delta_{k-1}^2-
\frac{(\tilde\rho_{k-1}^2-\tilde\sigma_{k-1}^2)^2}{\delta_{k-1}^2}\right) \\
&\le
\frac{1}{4}\left(2\tilde\rho_{k-1}^2+2\tilde\sigma_{k-1}^2-\tilde\rho_{k-1}^2-
\frac{(\tilde\rho_{k-1}^2-\tilde\sigma_{k-1}^2)^2}{\tilde\rho_{k-1}^2}\right) \\
&= \tilde\sigma_{k-1}^2 -\frac{\tilde\sigma_{k-1}^4}{4\tilde\rho_{k-1}^2}.
\end{align*}
Therefore,
\begin{align*}
\tilde\sigma_k^2 & = (\tilde\xi_k^*)^2 -\gamma_{k-1}+\gamma_k \\
&\le 
\tilde\sigma_{k-1}^2 -\frac{\tilde\sigma_{k-1}^4}{4\tilde\rho_{k-1}^2}
-\gamma_{k-1}+\gamma_k \\
& =
\tilde\sigma_{k-1}^2 -\frac{\tilde\sigma_{k-1}^4}{4\Vert\nabla f(\x_{k-1})\Vert^2/\ell^2}
-\gamma_{k-1}+\gamma_k \\
& \le 
\tilde\sigma_{k-1}^2 -\frac{\tilde\sigma_{k-1}^4}{4\Vert\nabla f(\x_{k-1})\Vert^2/\ell^2}
-\frac{\Vert\nabla f(\x_{k-1})\Vert^2}{L\ell} && \mbox{(by \eref{eq:gammadiff})} \\
& \le 
\tilde\sigma_{k-1}^2 -2\cdot\frac{\tilde\sigma_{k-1}^2}{2\Vert\nabla f(\x_{k-1})\Vert/\ell}
\cdot \frac{\Vert\nabla f(\x_{k-1})\Vert}{\sqrt{L\ell}} && \mbox{(since $a^2+b^2\ge 2ab$)} \\
& =
\tilde\sigma_{k-1}^2 \left(1-\sqrt{\frac{\ell}{L}}\right).
\end{align*}

In the other case, $\tilde\sigma_{k-1}^2/2\ge \tilde\rho_{k-1}^2$
so we obtain
\begin{align*}
\tilde\sigma_k^2 & = \tilde\rho_{k-1}^2 -\gamma_{k-1}+\gamma_k
&& \mbox{(by \eref{eq:sigmadef2})} \\
&= \frac{\Vert \nabla f(\x_{k-1})\Vert^2}{\ell^2}  -\gamma_{k-1}+\gamma_k\\
&\le 
\frac{\Vert \nabla f(\x_{k-1})\Vert^2}{\ell^2}-
\frac{\Vert \nabla f(\x_{k-1})\Vert^2}{L\ell} &&\mbox {(by \eref{eq:gammadiff})} \\
& = 
\frac{\Vert \nabla f(\x_{k-1})\Vert^2}{\ell^2}(1-\ell/L) \\
& = 
\tilde\rho_{k-1}^2(1-\ell/L) \\
&\le 
\tilde\sigma_{k-1}^2(1-\ell/L)/2 && \mbox{(by the hypothesis of the case).}
\end{align*}
It is a simple matter to confirm that $(1-\eps)/2\le (1-\sqrt{\eps})$
for any $\eps\in[0,1]$, thus establishing the theorem in this case.
\end{proof}

\section{Relationship between GD and IA}
\label{sec:GDanalysis2}

We already observed that GD and IA both work in the same
affine subspace $\mathcal{M}_k$ on each iteration.  In this section,
we develop further insight into their connection.
On each step IA solves two optimization
problems exactly, \eref{eq:ia2.xupd} and \eref{eq:ia2.yupd},
to obtain $\x_k$ and $\y_k$.  We argue
that GD computes optimal
solutions to these two problems not for the actual objective function $f$
but for some other objective function that agrees
with its partial information 
about $f$.
This is stated precisely in the following two theorems, the
first about $\x_k$ and the second about $\y_k$.

Regarding the computation of $\x_k$, let us recall that
GD takes a steepest descent step from $\x_{k-1}$ in
line \eref{eq:gd.steep} followed by a line search in
\eref{eq:gd.xupd}.  The simplest line-search method for
minimizing a convex function
is bisection based on the sign of the directional derivative
of the objective along the line.  Evaluation of signs of derivatives
does not give any information other than the location of the minimizer
since univariate convex functions are unimodal.  (In other words,
bisection to find the minimizer of two convex functions $f_1,f_2$
with a common minimizer will evaluate the same sequence
of points for $f_1$ as for  $f_2$.)    Therefore, the information
about $f$ used in GD to find $\x_k$
is the value of $\nabla f(\x_{k-1})$, the line
$\aff\{\obx_{k-1},\y_{k-1}\}$, and the location of the minimizer on that
line.  The following theorem says that GD chooses the optimal
$\x_k\in\mathcal{M}_k$
given this partial information about the objective function
$f$.  (In contrast, IA chooses the optimal $\x_k$ for the
true objective $f$.)

\begin{theorem}
  Given scalar $L>0$, a point $\x_{k-1}\in\R^n$, $n\ge 2$,
  a nonzero vector $\g\in\R^n$, define
  $\obx_{k-1}=\x_{k-1}-\g/L$.  Also, assume we are
  given
  a line $\Lambda \subset \R^n$ containing
  $\obx_{k-1}$ but not $\x_{k-1}$, and point $\x_{k}\in\Lambda$.
  Let $\mathcal{M}_k=\aff\{\Lambda,\x_{k-1}\}$.
  Assume that
  \begin{equation}
    \g^T(\x_k-\obx_{k-1})<0.
    \label{eq:angassum}
  \end{equation}
  (Assumption \eref{eq:angassum} will be explained later.)
  Define
  $\mathcal{F}$ to be:
  \begin{align*}
    \mathcal{F}=\{f:\R^n\rightarrow\R: &\mbox{$f$ is convex}, \\
    &\argmin\{f(\x):\x\in \Lambda\}=\x_{k}, \\
    &\mbox{$\nabla f$ is $L$-Lipschitz}, \\
    & \nabla f(\x_{k-1})=\g\}.
  \end{align*}
  Define $q(\x)=\sup\{f(\x)-f(\x_{k-1}):f\in\mathcal{F}\}$.
  Then 
  \begin{equation}
    q(\x_k)=\min\{q(\x):\x\in\mathcal{M}_k\}.
      \label{eq:s0argmin}
  \end{equation}
\end{theorem}
The class $\mathcal{F}$ is meant to capture the set of all functions whose
partial information known to the GD algorithm agrees with the partial information
of the actual objective function, that is, $\nabla f(\x_{k-1})=\g$ and
$\x_k$ is the minimizer of $f$ on the line $\Lambda$.  The conclusion
of the theorem is that for the worst $f$ in this class (the sup appearing
in the definition of $q$), $\x_k$ chosen by GD is optimal over $\mathcal{M}_k$.
The hypotheses on the given data correspond to the induction hypotheses of GD
except for \eref{eq:angassum}, which we discuss later.

\begin{proof}
  Since the algorithm is invariant under translation and rotation
  of space, without loss of generality we can 
  transform coordinates to identify $\mathcal{M}_k$ with
  the $(x_1,x_2)$-plane. 
  Assume that the transformation places
  the point $\obx_{k-1}$ at the origin $(0,0)$,
  the line $\Lambda$ on the $x_1$-axis, the point
  $\x_k$ at $(s,0)$.
  Denote the point $\x_{k-1}$ with $(a,b)$ in this plane.
  Because of the identification of $(0,0)$ with $\obx_{k-1}$,
  we know that for $f\in\mathcal{F}$, $\g=\nabla f(\x_{k-1})=\nabla f(a,b)=(aL,bL)$ 
  in this rotated and translated coordinate system.  
  Since the length $\g$ is unchanged by translation and rotation,
  then $(a,b)$ must satisfy the restriction that $((aL)^2+(bL)^2)^{1/2}$ is the
  original length $\Vert\g\Vert$.  Observe that
  $f(0,0)-f(a,b)\le -\Vert\g\Vert^2/(2L)=-(L/2)(a^2+b^2)$
  by \eref{eq:fdesc0}.
  For the remainder of the proof, let $k_0=-(L/2)(a^2+b^2)$.
  Thus, $f(s,0)-f(a,b)\le k_0$
  since $(s,0)$ is the minimizer of $f\in\mathcal{F}$ over $\Lambda$
  (hence achieves a value
  less than $f(0,0$)).  This inequality holds for all $f\in\mathcal{F}$,
  thus showing that $q(\x_k)\le k_0$.

  In these transformed coordinates, assumption \eref{eq:angassum}
  is rewritten as the inequality $as<0$.
  For the remainder of this proof, assume $a>0$ hence $s<0$; the
  other case is obtained by reflection of the $x_1$-coordinate.

  Consider the function,
  $$\hat f_0(x,y) = \left\{
  \begin{array}{ll}
    Lx^2/2 + Ly^2/2,& x\ge 0, \\
    Ly^2/2, & x\le 0,
  \end{array}
  \right.
  $$
  The level curves of this function are semicircles in the right half-plane
  and parallel rays in the left half-plane.  It satisfies all the conditions
  for membership in $\mathcal{F}$ except that $(s,0)$ is not the unique minimizer
  over $\Lambda$;
  all points of the ray $\{(x,0):x\le 0\}$ are minimizers. Therefore,
  we perturb this function slightly.

  Fix a $\delta>0$ small.
  Define scalars $m=-\delta (s+a)/(L-2\delta)$ and $p=\delta m - \delta s$.
  Note that
  $m,p$
  tend to 0 as $\delta\rightarrow 0$, so assume that $\delta$ is
  sufficiently small that $2\delta<L$, $|m|<|s|$, $|m|<|a|$.
  Consider the following function $\hat f_\delta$:
  $$\hat f_\delta(x,y) = \left\{
  \begin{array}{ll}
    q_1(x-m)^2/2 + p(x-m)+Ly^2/2,& x\ge m, \\
    q_2(x-m)^2/2+p(x-m) +Ly^2/2, & x\le m,
  \end{array}
  \right.
  $$
  where $q_1=L-\delta$, $q_2=\delta$.
  
  It is straightforward to check that $\hat f_\delta \in\mathcal{F}$.
  In particular, $\partial \hat f_\delta(s,0)/\partial x=\delta(s-m)+p=0$ since
  $p=\delta(m-s)$.
  The choice of $m$ ensures that $\partial \hat f_\delta(a,b)/\partial x=aL$.
  Observe that $\hat f_\delta(s,0)-\hat f_\delta(a,b)\rightarrow k_0$ as $\delta\rightarrow 0$,
  and $(s,0)$ is the minimizer of $f_\delta$.  Therefore,  for all 
  $\x\in\mathcal{M}_k$, $f_\delta(\x)-f_\delta(\x_{k-1})$ is bounded
  below by $k_0$ plus a residual that tends to 0 as $\delta\rightarrow 0$.
  which means
  that for all $\x\in\mathcal{M}_k$, $q(\x)\ge k_0$ (since $f_\delta$ is a candidate
  for the supremum in the definition of $q(\cdot)$).
  Since we already established that $q(\x_k)\le k_0$, this proves the theorem.
\end{proof}

Let us know examine assumption \eref{eq:angassum}, which
in transformed coordinates is written $as<0$.  If
$as>0$, then the minimizer over $\mathcal{M}_k$
cannot lie on the $x$-axis, i.e.,
$\mathcal{F}$ is empty as the following argument shows.
(The case of $as=0$ needs separate
treatment, which we omit).  Consider an arbitrary $f\in \mathcal{F}$.
Select $t$ to solve the following equation:
\begin{equation}
  ((s+at,bt) - (a,b))^T(a,b)=0;
  \label{eq:saborth}
\end{equation}
one easily determines that
\begin{equation}
  t=\frac{a^2+b^2-as}{a^2+b^2}.
  \label{eq:tformula}
\end{equation}
It is impossible that $a^2+b^2-as<0$, or equivalently, that $t<0$ because then
\begin{align*}
  f(s,0)&\ge f(a,b)+\nabla f(a,b)^T((s,0)-(a,b))  && \mbox{(by the subgradient inequality)}
  \\
  & = f(a,b)+(aL,bL)^T((s,0)-(a,b)) \\
  &= f(a,b)+L(as-a^2-b^2) \\
  & > f(a,b),
\end{align*}
contradicting the minimality of $(s,0)$.

But there is also a contradiction when $t\ge 0$.
Consider the following chain of inequalities:
\begin{align*}
  f(s,0)+\frac{Lt^2\Vert(a,b)\Vert^2}{2}&\ge f(s+at,bt) &&\mbox{(by $L$-smoothness} \\
  &&& \mbox{since $\nabla f(s,0)=(0,0)$)} \\
  &\ge f(a,b)+\nabla f(a,b)^T((s+at,bt)-(a,b))\\
     &&& \mbox{(the subgradient inequality)} \\
  &=f(a,b) && \mbox{(using \eref{eq:saborth})} \\
    &\ge f(0,0)+L/2\Vert(a,b)\Vert^2 && \mbox{(by \eref{eq:fdesc0})} \\
    & \ge f(s,0)+L/2\Vert(a,b)\Vert^2 && \mbox{(since $(s,0)$ is the minimizer)}.
\end{align*}
This chain of inequalities starts and ends at the same quantity except for the
presence of $t^2$; thus, the inequalities can hold only if $t\ge 1$.  But
if $as>0$ and $a^2+b^2-as\ge 0$,
then it follows from \eref{eq:tformula} that $0\le t <1$.  So we conclude
that $as>0$ contradicts the optimality of $(s,0)$.
More comments on this matter appear at the end of this section.

The next theorem covers $\y_k$ and requires a different construction.
Its format and interpretation are analogous to the previous
theorem.
\begin{theorem}
  Suppose the following data is given: two points $\x_{k-1},\y_{k-1}$ in
  $\R^n$, $n\ge 3$, 
  a nonzero vector $\g\in\R^n$, four positive scalars $\rho,\sigma,\delta,\ell$
  that satisfy the following conditions.  For
  the remainder of the theorem, let $\oobx_{k-1}=\x_{k-1}-\g/\ell$.
  The conditions are:
  \begin{itemize}
  \item
    $\delta=\Vert\oobx_{k-1}-\y_{k-1}\Vert$,
  \item
    $\rho<\Vert\g\Vert/\ell$
  \item
    $\rho+\sigma\ge \delta$,
  \item
    $\delta\ge \sqrt{|\rho^2-\sigma^2|}$, and
  \item
    \begin{equation}
    (\y_{k-1}-\oobx_{k-1})^T\g > \rho^2\ell/\lambda^*,
      \label{eq:yxbbcond}
    \end{equation}
    where $\lambda^*$ is given by \eref{eq:lambdastar}.
    (Assumption \eref{eq:yxbbcond} will be explained  later.)
  \end{itemize}
  Let $\mathcal{F}$ be the following set of functions:
  \begin{align}
    \mathcal{F}=\{f:\R^n\rightarrow\R:
    & \mbox{$f$ is strongly convex with modulus $\ell$}, \notag \\
    & \nabla f(\x_{k-1})=\g, \notag \\
    & \frac{\Vert\g\Vert^2}{\ell^2}-\frac{2(f(\x_{k-1})-\min\{f\})}{\ell}=\rho^2,
    \mbox{ and} \label{eq:rhocond2}\\
    & \argmin\{f\}\in B(\y_{k-1},\sigma)\}. \notag
  \end{align}
  Here, $\min\{f\}$ is shorthand for $\min\{f(\x):\x\in\R^n\}$ and
  similarly for $\argmin\{f\}$.
  Let $\mathcal{M}_k$ denote $\aff\{\x_{k-1},\y_{k-1},\oobx_{k-1}\}$.
  Define
  $$q(\y)=\sup\{\Vert\y-\argmin\{f\}\Vert: f\in\mathcal{F}\}.$$
  Let $\y_k$ be the point computed by the GD algorithm for this data
  using \eref{eq:gd.yupd}.
  Then
  \begin{equation}
    \y_k=\argmin\{q(\y):\y\in \mathcal{M}_k\}.\label{eq:yksolve}
  \end{equation}
\end{theorem}

\begin{proof}
  By translating and rigidly rotating space, we can identify $\mathcal{M}_k$ with
  the $(x_1,x_2)$ plane.  
  After this coordinate transformation,
  we may assume $\oobx_{k-1}$ coincides with the origin $(0,0)$.
  For the remainder of the discussion, remaining coordinates $x_3,x_4,\ldots,x_n$\
  are 
  not written and are assumed to be 0's.
  Let us write $\x_{k-1}=(a,b)$.
  Since $\oobx_{k-1}$ is at the origin, this implies $\g=(a\ell,b\ell)$.
  Choose the transformation to make
  $\y_{k-1}$ lie on the positive $x_1$-axis at position $(\delta,0)$.
  (Recall that $\delta=\Vert\y_{k-1}-\oobx_{k-1}\Vert$.)
  
  Let $\lambda^*$ be as in \eref{eq:lambdastar}. As in \eref{eq:gd.yupd},
  let  $\y_{k}=(1-\lambda^*)\oobx_{k-1}+\lambda^*\y_{k-1}$, which in
  this coordinate system is 
  $\y_k=(\delta\lambda^*, 0).$
  
  Let $f$ be an arbitrary member of $\mathcal{F}$. The hypotheses
  on $\mathcal{F}$ imply that 
  $\argmin\{f\}\in B(\oobx_{k-1},\rho)\cap B(\y_{k-1},\sigma)$, and therefore,
  by Lemmas \ref{lem:ballintersect1} 
  and \ref{lem:ballintersect2}, $\Vert\argmin\{f\}-\y_k\Vert\le \xi^*$,
  where $\xi^*$ is given by \eref{eq:xistardef}.  This shows that
  $q(\y_{k})\le \xi^*$.

  The remainder of the proof shows that
  $q(\y)>\xi^*$ for $\y\in \mathcal{M}_k-\{\y_k\}$, which will
  establish \eref{eq:yksolve}.   Let
  $\kappa$ stand for either $+1$ or $-1$.  Embed
  $\mathcal{M}_k$ in one higher dimension and write 3-tuples of
  coordinates (so that $\x_{k-1}=(a,b,0)$, 
  $\g=(a\ell,b\ell,0)$, $\y_{k-1}=(\delta,0,0)$ and
  so forth).  Define
  \begin{equation}
  \x^*=\left(\begin{array}{c}
    \delta\lambda^* \\
    0 \\
    \kappa \xi^*
  \end{array}
  \right).
  \label{eq:xoptdef1}
  \end{equation}
  Define 
  $$f(\x)=(\ell/2)\Vert\x-\x^*\Vert^2 +\ell\left|(\x^*)^T(\x-\x^*)\right|.$$
  First, observe the obvious consequences of this
  formula that $\x^*=\argmin\{f\}$ and that $f(\x^*)=0$.
  We claim that $f\in\mathcal{F}$.  The fact that $f$ is 
  $\ell$-strongly convex follows from the presence of the
  first term.  The second term is convex but not strongly convex.

  In order to establish the remaining conditions for
  membership in $\mathcal{F}$, we
  first determine which branch of the absolute value
  holds when evaluating $f(\x_{k-1})$; in particular,
  we establish the inequality that $(\x^*)^T(\x_{k-1}-\x^*)> 0$,
  i.e., that $(\x^*)^T\x_{k-1}>\Vert\x^*\Vert^2$.  The left-hand side
  evaluates to $a\delta\lambda^*$, while the right-hand side
  evaluates to $(\delta\lambda^*)^2+(\xi^*)^2$
  which simplifies to $\rho^2$ according to
  \eref{eq:lambdastar} and \eref{eq:xistardef}.  
  Thus, we must establish $a\delta\lambda^*>\rho^2$; this follows
  from \eref{eq:yxbbcond} which states that
  $\delta a\ell>\rho^2\ell/\lambda^*$ in the transformed coordinates.

  This inequality implies that in the neighborhood of
  $\x_{k-1}$, the absolute value sign appearing in the definition
  of $f$ may be dropped.
  We now establish
  the remaining conditions for membership of $f$ in $\mathcal{F}$.
  We have $\nabla f(\x_{k-1})=\ell(\x_{k-1}-\x^*)+\ell\x^*=\ell\x_{k-1}=\g$,
  so the second condition is established. 
  For the third condition, recalling that $(\x^*)^T\x_{k-1}=a\delta\lambda^*$
  while $\Vert\x^*\Vert^2=\rho^2$, we compute,
  \begin{align*}
    f(\x_{k-1}) &=(\ell/2)((a-\delta\lambda^*)^2+b^2+(\xi^*)^2) + \ell(a\delta\lambda^*-\rho^2) \\
    &=(\ell/2)(a^2 + b^2 -\rho^2),
  \end{align*}
  where, to obtain the second line, we combined like terms and used
  the already established identity  $(\delta\lambda^*)^2+(\xi^*)^2=\rho^2$.
  Therefore,
  \begin{align*}
    \frac{\Vert\g\Vert^2}{\ell^2} - \frac{2}{\ell}(f(\x_{k-1})-f(\x^*))
    &= (a^2+b^2) - (a^2+b^2-\rho^2) \\
    &=\rho^2,
  \end{align*}
  thus establishing \eref{eq:rhocond2}.  For the last condition,
  \begin{align*}
    \Vert\y_{k-1}-\x^*\Vert^2 &= (\delta-\delta\lambda^*)^2+(\xi^*)^2 \\
    &= (\sigma^*)^2.
  \end{align*}

  Thus, membership of $f\in\mathcal{F}$ is established.  Note that
  the two choices for $\kappa$, namely, $\pm 1$ 
  lead to two different optimizers $\x^*$ in \eref{eq:xoptdef1}
  whose distance apart is
  $2\xi^*$.  Therefore, the midpoint of the two optimizers is the
  only point in $\R^n$ whose distance from both of them is
  bounded above by $\xi^*$.  This midpoint is exactly $\y_k$.  This
  proves that $q(\y_k)\ge \xi^*$ (we already showed that $q(\y_k)\le \xi^*)$),
  and also that for any $\y\in\mathcal{M}_k-\{\y_k\}$, $q(\y)>\xi^*$.
  This establishes \eref{eq:yksolve}.
\end{proof}  

We now discuss the conditions imposed on the given data.
The condition $\delta=\Vert\oobx_{k-1}-\y_{k-1}\Vert$ is
simply part of the construction, and the conditions
$\rho<\Vert\g\Vert/\ell$ and
$\rho+\sigma\ge \delta$ ensure that $\mathcal{F}$ is
nonempty.  The condition
$\delta\ge \sqrt{|\rho^2-\sigma^2|}$ is the main case of
the two cases arising in the main loop of GD.
In particular, the condition \eref{eq:gd.ifstmt} is used to establish
\eref{eq:gd.chain1}.  We omit
the treatment of the minor case.

Finally, we discuss \eref{eq:yxbbcond}.  If a strengthened version of this
condition fails to hold,  then
$\mathcal{F}$ is empty.
In particular, strong convexity implies that the optimizer lies in a ball
of radius $\Vert\g\Vert/\ell$ about $\x_{k-1}$, i.e.,
$$\Vert\x_{k-1}-\x^*\Vert^2 \le \Vert\g\Vert^2/\ell^2.$$
In the
transformed coordinates, this is written
$$(a-\delta\lambda^*)^2+b^2+(\xi^*)^2\le a^2+b^2,$$
which simplifies to
$$2a\delta\lambda^*\ge\rho^2,$$
while \eref{eq:yxbbcond} is written as $a\delta\lambda^*>\rho^2$ in the
transformed coordinates, so the
two bounds differ by a factor of 2.  In the case that
$\rho^2\in [a\delta\lambda^*,2a\delta\lambda^*]$, the two possibilities are (1)
$\mathcal{F}=\emptyset$, or (2) there is an $f\in\mathcal{F}$, but
the construction of $f$ used in the proof
does not work.  We do not know
which possibility is correct.

We now summarize the results in this section with some observations.
We have shown that the computations of $\x_k,\y_k$ in the GD algorithm
solve minimization problems akin to those in IA except that
the minimizer pertains to the worst case $f$ that agrees with
the partial information that GD has about $f$ rather than the true $f$.

In both cases, the theorems had some apparently extraneous assumptions,
namely \eref{eq:angassum} in the first theorem, and \eref{eq:yxbbcond} in the second.
However, in both these cases, the extraneous assumptions indicate
that GD may not be using the
information about $f$ entirely.  In the case of the first theorem,
if $\g^T(\x_{k}-\obx_{k-1})>0$
then the minimizer of $f$ over $\mathcal{M}_k$ cannot be on the
line searched by GD, yet GD does not use this information.
In the case of the second theorem, GD does not
use the fact that $\x^*\in  B(\x_{k-1},\Vert \nabla f(\x_{k-1})\Vert/\ell)$.

The main point of this section is to clarify the relationship between IA
and GD, but the preceding paragraph reveals a second point.  By constructing
these example functions to show that the GD algorithm is the best possible
in some cases but not others
given the limited information that it uses,
we also show that it may be possible to improve
on GD by making better use of the information
(such as the third ball mentioned in the
previous paragraph) in certain cases.
We do not pursue this idea here, but
see, e.g., \cite{Drusvyatskiy} for results in this direction.

\section{Computability of the potential}
\label{sec:computable}

A point to make about the GD algorithm and its
analysis is that the potential $\tilde\sigma_k^2$ is
computable on every iteration of the algorithm, i.e.,
it does not require prior knowledge of $\x^*$.  (It does,
however, require prior knowledge of $\ell$.)
This fact
can also be deduced from the original
presentation of \cite{bubeck}, although the computability
is not further used therein.

In this section we explain in more detail what we
mean by ``computable''. In particular, a potential
$\tilde\sigma_k$ that is an upper bound to \eref{eq:psidef}
is computable if it has the
following properties.

\begin{enumerate}
\item
  It must be possible to compute the potential without
  prior knowledge of $\x^*$ or $f(\x^*)$.
  This is the reason that we regard the potential
  $\Psi_k$ itself defined in \eref{eq:psidef} as
  noncomputable.
\item
  It must have an {\em a posterior} dependence on the actual convergence
  of the algorithm, in the sense that if
  $\x_k,\y_k$ are very close to $\x^*$, then it
  should be the case that $\tilde\sigma_k$ is
  very close to zero.  In particular, this
  rules out using a completely {\em a priori} potential like
  $$\tilde\sigma_k := C\Vert\nabla f(\x_0)\Vert^2\left(1-\sqrt{\frac{\ell}{L}}\right)^k$$
  for a fixed constant $C$.  Although this potential indeed
  is an upper bound on \eref{eq:psidef} and is computable, it has no relationship
  to the current iterate and therefore has no algorithmic use.
\item
  The potential should decrease by the factor $(1-\sqrt{\ell/{L}})$
  per iteration (or perhaps $(1-\mbox{const}\cdot\sqrt{\ell/{L}})$, since a constant
  factor improvement may be possible).  Thus, although
  $C\Vert \nabla f(\x_k)\Vert^2$ for a correctly chosen $C$ is
  an upper bound on \eref{eq:psidef}, it does not satisfy
  our requirement of steady decrease and in fact can be
  oscillatory.

  The significance of guaranteed decrease in the potential is twofold.
  First, the guaranteed decrease is useful for theoretical analysis
  to establish a linear convergence rate.  Indeed, it is used
  herein for this
  purpose to establish previously known convergence rates for AG
  and CG in a new manner.

  Second, steady decrease in the potential can be used in an
  algorithm to ensure that progress is being made.  To give
  one example not pursued herein, consider the problem of detecting stagnation
  due to imprecise arithmetic in linear conjugate gradient.
  Although
  $\Vert \nabla f(\x_k)\Vert^2$ is commonly used as a termination
  criterion for linear conjugate gradient, it is not suitable
  for use as a stagnation test because it can be highly oscillatory
  and therefore cannot be used to check whether a single
  iteration was successful.
  A steadily decreasing potential, however, could be used
  in a CG stagnation test.  (We have preliminary results
  on this matter that will be the subject of future work.)

  Our present algorithmic use of the potential as a measure of steady
  decrease is reported in Sections~\ref{sec:hybrid}--\ref{sec:comp}.  
  We present an experiment to illustrate why
  $\Vert \nabla f(\x_k)\Vert^2$ is not a suitable substitute for
  $\tilde\sigma_k$ in Section~\ref{sec:comp}.
\end{enumerate}

The potential for GD given by
\eref{eq:sigmadef} and \eref{eq:sigmadef2}.  As mentioned
earlier, it is a slight variant of the potential defined
by the authors of GD; theirs also has these properties.
It is somewhat surprising that the same potential also
applies to conjugate gradient and accelerated
gradient, as developed in the next few sections.
We do not know of any other computable potential with these
properties.

\section{Analysis of linear conjugate gradient}
\label{sec:CGanalysis}

The linear conjugate gradient (CG) algorithm for 
minimizing $f(\x)=\x^TA\x/2 -\b^T\x$, 
where
$A$ is a symmetric positive definite matrix, 
is due to Hestenes and Stiefel
\cite{hestenesstiefel}
and is as follows.
\begin{align}
& \mbox{\bf Linear Conjugate Gradient} \notag \\
& \x_0:=\mbox{arbitrary} \notag \\
& \r_0:=\b-A\x_0 \notag \\
& \p_1:=\r_0 \notag \\
& \mbox{for } k:=1,2,\ldots, \notag \\
&\displaystyle \hphantom{\mbox{for }} 
\alpha_{k}:= \frac{\r_{k-1}^T\r_{k-1}}{\p_{k}^TA\p_{k}} \label{eq:alphadef} \\
&\hphantom{\mbox{for }} 
\x_{k} := \x_{k-1}+\alpha_{k}\p_{k} \label{eq:cg.xupd} \\
&\hphantom{\mbox{for }} 
\r_{k} := \r_{k-1}-\alpha_{k} A\p_{k} \label{eq:rupd} \\
&\displaystyle \hphantom{\mbox{for }} 
\beta_{k+1}:= \frac{\r_{k}^T\r_{k}}{\r_{k-1}^T\r_{k-1}} \label{eq:betadef}\\
&\hphantom{\mbox{for }} 
\p_{k+1} :=  \beta_{k+1}\p_{k}+\r_{k} \label{eq:pupd} \\
&\mbox{end} \notag
\end{align}

We now show that linear CG exactly implements
Algorithm IA, and therefore also satisfies the bound of Theorem~\ref{thm:iacvg}.
This is perhaps surprising because CG does not have prior
information about $\x^*$.  The following key results about CG are from the
original paper:

\begin{theorem}(\cite{hestenesstiefel})
Let $\mathcal{V}_k=\x_0+\Span\{\r_0,\ldots,\r_{k-1}\}$ in CG.
Then

(a) An equivalent
formula is $\mathcal{V}_k=\x_0+\Span\{\p_1,\ldots,\p_{k}\}$,

(b) $\x_k$ is the minimizer of $f(\x)$ over $\mathcal{V}_k$,

(c) $\r_k=-\nabla f(\x_k)$, and

(d) $\x_k+\tau_k\p_k$ is the minimizer of $\Vert \x-\x^*\Vert$ over
$\mathcal{V}_k$, where
\begin{equation}
\tau_k = \frac{2(f(\x_k)-f(\x^*))}{\Vert\r_{k-1}\Vert^2}.
\label{eq:tau_k}
\end{equation}
\label{thm:hest}
\end{theorem}

Part (b) appears as Theorem 4:3 of \cite{hestenesstiefel},
while parts (a) and  (c) are not stated explicitly.  All of (a)--(c)
are covered
by most textbook treatments of CG.  Part (d) appears as Theorem 6:5 and
is less well known.

The following theorem establishes the claim that Algorithm CG implements
IA.  The principal result is part (b).  The remaining parts
are necessary to support the induction proof.

\begin{theorem}
Suppose IA and CG are applied to the same quadratic function with
the same $\x_0$.  Let the sequences of iterates be denoted
$\x_k^{\rm IA}$ and $\x_k^{\rm CG}$ respectively.
Then for each $k=1,2,\ldots$,

(a) $\mathcal{M}_{k}\subset \mathcal{V}_k$,

(b) $\x_k^{\rm CG}=\x_k^{\rm IA}$, and

(c) $\y_k=\x_{k}^{\rm CG}+\tau_{k}\p_{k}$.

\end{theorem}

\begin{proof}
For the $k=1$ case, observe that $\p_1=-\nabla f(\x_0)$ so
$\mathcal{M}_1=\mathcal{V}_1$. Since $\x_1^{\rm CG}$ minimizes
$f(\x)$ over $\mathcal{V}_1$ while $\x_1^{\rm IA}$ minimizes
$f(\x)$ over $\mathcal{M}_1$, we conclude $\x_1^{IA}=\x_1^{\rm CG}$.
For (c), observe that $\y_1$ minimizes $\Vert\y-\x^*\Vert$
over $\mathcal{M}_1$ by \eref{eq:ia2.yupd}, while $\x_1^{\rm CG}+\tau_1\p_1$
minimizes the same function over the same affine space, so (c)
is established.

Now assuming (a)--(c) hold for some $k\ge 1$, we establish them for $k+1$.
We will write $\x_k$ for both $\x_k^{\rm CG}$ and  $\x_k^{\rm IA}$
since these
are equal by induction.  For (a), we start with
$\mathcal{M}_{k+1}=\x_k+\Span\{\x_k-\y_k,\nabla f(\x_k)\}$.
We already know from (c) that $\y_k-\x_k=\tau_k\p_k\in{\bf T}\mathcal{V}_k
\subset {\bf T}\mathcal{V}_{k+1}.$  Also,
$\nabla f(\x_k)=-\r_k=\beta_{k+1}\p_k-\p_{k+1}$ (by \eref{eq:pupd}), hence
$\nabla f(\x_k)\in {\bf T}\mathcal{V}_{k+1}$
Thus, ${\bf T}\mathcal{M}_{k+1}\subset{\bf T}\mathcal{V}_{k+1}$,
so showing $\mathcal{M}_{k+1}\subset\mathcal{V}_{k+1}$ is reduced to finding
a single common point, and we may take $\x_k$
to be this point.

For (b), $\x_{k+1}^{\rm CG}$ minimizes $f(\x)$ over $\mathcal{V}_{k+1}$
by Theorem~\ref{thm:hest}(a).
We also know that $\x_{k+1}^{\rm CG}\in \mathcal{M}_{k+1}$
because 
\begin{align*}
\x_{k+1}^{\rm CG}&=\x_k+\alpha_{k+1}\p_{k+1}  && \mbox{(by \eref{eq:cg.xupd})}\\
&=\x_k+\alpha_{k+1}(\r_k+\beta_{k+1}\p_k)  && \mbox{(by \eref{eq:pupd})}\\
&=\x_k+\alpha_{k+1}(-\nabla f(\x_k) + \beta_{k+1}\p_k) 
&& \mbox{(by Theorem~\ref{thm:hest}(b))}\\
&=\x_k+\alpha_{k+1}(-\nabla f(\x_k) + \beta_{k+1}(\y_k-\x_k)/\tau_k) &&
\mbox{(by induction, part (c))} \\
&\in \x_k+\Span\{\nabla f(\x_k), \y_k-\x_k\}.
\end{align*}
Since $\mathcal{M}_{k+1}\subset\mathcal{V}_{k+1}$ according to part (a), the
optimality of $\x_{k+1}^{\rm CG}$ with respect to $\mathcal{V}_{k+1}$ implies
that it is also optimal for $f(\x)$ with respect to $\mathcal{M}_{k+1}$,
hence $\x_{k+1}^{\rm IA}=\x_{k+1}^{\rm CG}.$  Thus, write $\x_{k+1}$ for 
both vectors for the remainder of the argument.

A similar argument establishes (c).  First, we observe
that $\x_{k+1}+\tau_{k+1}\p_{k+1}$ lies in $\mathcal{M}_{k+1}$ 
because $\x_{k+1}+\tau_{k+1}\p_{k+1}=\x_k+(\alpha_{k+1}+\tau_{k+1})\p_{k+1}$,
and the we can proceed as in the last paragraph except with 
$(\alpha_{k+1}+\tau_{k+1})$ playing the role of $\alpha_{k+1}$.  Next,
$\x_{k+1}+\tau_{k+1}\p_{k+1}$ minimizes $\Vert\x-\x^*\Vert$ over $\mathcal{V}_{k+1}$
by Theorem~\ref{thm:hest}(d).  Thus, $\x_{k+1}+\tau_{k+1}\p_{k+1}$ must
be the same as $\y_{k+1}$.  This concludes the induction.
\end{proof}

The surprising aspect of this analysis is that CG exactly
identifies $\mathcal{M}_k$ that appears in Algorithm IA despite
not ever computing $\y_k^{\rm IA}$.  The reason is that the line
$\aff\{\x_k,\y_k^{\rm IA}\}$ agrees with the
line  $\x_k+\Span\{\p_k\}$, which is computed by CG.

\section{A computable potential for linear conjugate gradient}
\label{sec:CGanalysis2}

The analysis in the preceding section shows that CG implements
the idealized algorithm.  However, the decrease in the potential
cannot be measured during the algorithm because \eref{eq:tau_k} requires
prior knowledge of the optimizer.  In this section, we observe
that the GD 
potential can also be used for CG, yielding
a computable potential.
An application of
this potential will be presented in Section~\ref{sec:hybrid}.   

We define an auxiliary sequence of vectors $\y_k$
using the formulas in GD.  This sequence is not the
true minimizer that occurs in IA
\eref{eq:ia2.yupd} and in
Theorem~\ref{thm:hest}, part (d). But nonetheless,
$\Vert \y_k-\x^*\Vert$ shrinks sufficiently fast
to establish the necessary decrease in the potential.

In particular, we exactly mimic the equations that
define the quantities
$\tilde\sigma_k$, $\tilde\rho_k$, $\lambda_k$, $\delta_k$, $\y_k$
in GD, and modify only $\x_k$ so that it is computed using
the CG algorithm instead of the GD algorithm.
The two same two theorems that held for GD also hold for CG:

\begin{theorem}
For each $k=0,1,2,\ldots$,
$$\Vert\x^*- \y_k\Vert^2 + \frac{2(f(\x_k)-f(\x^*))}{\ell}\le \tilde\sigma_k^2.$$
\end{theorem}

\begin{theorem}
For each $k=1,2,\ldots$,
$$\tilde\sigma_{k}^2\le \left(1-\sqrt{\frac{\ell}{L}}\right)\tilde\sigma_{k-1}^2.$$
\end{theorem}

The following observations about $\x_k$
computed in CG show that the same proofs of the
previous theorems work for CG.

First, $\y_{k}\in\mathcal{V}_{k}$, whereas $\r_{k}$ is orthogonal
to ${\bf T}\mathcal{V}_{k}$ (a well known property of CG), and thus 
$\nabla f(\x_{k})^T(\y_{k}-\x_{k})=0$.
This is \eref{eq:gradorth}, which was a necessary ingredient
in the proof of GD.

Second, \eref{eq:gammadiff} still holds because 
the CG step from $\x_{k-1}$ to $\x_k$ improves
$f$ at least as much as the step from  $\x_{k-1}$
to $\obx_{k-1}$, since $\obx_k$ lies in the Krylov space 
$\mathcal{V}_k$
where $\x_k$ is optimal for $f$ over this space.
 
\section{Accelerated gradient}
\label{sec:AGanalysis}

The Accelerated Gradient (AG) algorithm of Nesterov is an even looser approximation
to IA than GD in the sense that there is no optimization subproblem per iteration;
instead, all step lengths are fixed.   
For this section, let us define
\begin{equation}
\kappa = \frac{L}{\ell},
\label{eq:kappadef}
\end{equation}
because this ratio, sometimes called the
{\em condition number} of $f$, is used often throughout the algorithm and
analysis.

The algorithm is as follows.  

\begin{align}
& \mbox{\bf Accelerated Gradient} \notag \\
& \x_0:=\mbox{arbitrary} \notag \\
& \w_0:=\x_0 \notag \\
& \mbox{for } k:=1,2,\ldots, \notag \\
& \hphantom{\mbox{for }} \x_k := \w_{k-1}-\nabla f(\w_{k-1})/L \label{eq:wupd} \\
& \hphantom{\mbox{for }} \w_k := \x_k + \theta(\x_k-\x_{k-1})\mbox { where } \label{eq:ag.xupd} \\
&  \hphantom{\mbox{forforforfor}} \displaystyle \theta=
       \frac{\sqrt{\kappa}-1}{\sqrt{\kappa}+1}. \label{eq:thetadef} \\
&\mbox{end} \notag
\end{align}

Note that some versions of AG in the literature
vary the choice of $\theta$
(e.g., see \cite{Nesterov:book}).

For the purpose of analysis, let us define the following auxiliary sequences
of vectors and scalars:
\begin{align}
\oobw_k &= \w_k-\nabla f(\w_k)/\ell, && (k=0,1,\ldots)
\label{eq:ag.oobw} \\
\y_0 &=\x_0, \notag \\
\y_k &= \x_k+\tau(\x_k-\x_{k-1}),
   && (k=1,2, \ldots) \label{eq:ag.ydef} \\
\tilde\sigma_0 &= \sqrt{2}\Vert\nabla f(\x_0)\Vert /\ell, \label{eq:ag.tsigma0}\\
\tilde\sigma_{k+1} &= 
\bigg[(1-\kappa^{-1/2})\tilde\sigma_{k}^2+ \frac{2(f(\x_{k+1})-f(\w_{k}))}{\ell}
\notag\\
&\hphantom{=}
\quad\mbox{}+\frac{\Vert\nabla f(\w_{k})\Vert^2}{L\ell}-(\kappa^{1/2}-\kappa^{-1/2})\Vert\w_{k}-\x_{k}\Vert^2\bigg]^{1/2}
&& (k=0,1,2,\ldots) \label{eq:ag.tsigma}
\end{align}
where
\begin{equation}
\tau= \sqrt{\kappa}-1.
\label{eq:taudef}
\end{equation}

We prove two main results about these
scalars.  The first shows that
$\tilde\sigma_k$ is decreasing at the appropriate
rate, while the second shows that it is
an upper bound on the distance to the optimizer.

\begin{theorem}
For each $k=0,1,2,\ldots,$
\begin{equation}
\tilde\sigma_{k+1}^2 \le (1-\kappa^{-1/2})\tilde\sigma_{k}^2.
\label{eq:agthm1}
\end{equation}
\end{theorem}

\begin{proof}
By squaring both sides of \eref{eq:ag.tsigma}, it is apparent that
\eref{eq:agthm1} reduces to showing:
$$\frac{2(f(\x_{k+1})-f(\w_{k}))}{\ell} +
\frac{\Vert\nabla f(\w_{k})\Vert^2}{L\ell}-(\kappa^{1/2}-\kappa^{-1/2})\Vert\w_{k}-\x_{k}\Vert^2\le 0.$$
Clearly it suffices to show
$$\frac{2(f(\x_{k+1})-f(\w_{k}))}{\ell} +
\frac{\Vert\nabla f(\w_{k})\Vert^2}{L\ell}\le 0.$$
This follows immediately from \eref{eq:wupd}, which implies
that $f(\x_{k+1})\le f(\w_k)-\Vert\nabla f(\w_k)\Vert^2/(2L)$.
\end{proof}

\begin{theorem}
For each $k=0,1,2\ldots$,
\begin{equation}
\Vert\y_k-\x^*\Vert^2 +\frac{2(f(\x_k)-f(\x^*))}{\ell}\le
\tilde\sigma_k^2.
\label{eq:ag.tsigmabd1}
\end{equation}
\end{theorem}

\begin{proof}
The proof of \eref{eq:ag.tsigmabd1} is by induction on $k$.  
We start by deriving some preliminary relationships.
It is clear from \eref{eq:ag.xupd} and \eref{eq:ag.ydef} that $\w_k,\x_k,\y_k$
are collinear and the tangent to their common line is $\x_k-\x_{k-1}$, hence
we easily obtain from these equations:
\begin{align}
\y_k&=\x_k+\frac{\tau}{\theta}(\w_k-\x_{k}) &&\mbox{(by \eref{eq:ag.xupd} and \eref{eq:ag.ydef})} \label{eq:wxy0} \\
&=\x_k+(\sqrt{\kappa}+1)(\w_k-\x_k) &&\mbox{(by \eref{eq:thetadef} and \eref{eq:taudef})}\notag \\
&=(\sqrt{\kappa}+1)\w_k - \sqrt{\kappa}\x_k. \label{eq:wxy}
\end{align}

For the $k=0$ case,
\eref{eq:ag.tsigmabd1} follows from the initialization in
\eref{eq:ag.tsigma0} and strong convexity.

We now assume the result \eref{eq:ag.tsigmabd1} holds for
$k$ and establish it for $k=1$.
The proof relies on Lemma~\ref{lem:ballintersect1}, so first we must
argue that $\y_{k+1}$ lies on the
line segment between $\y_k$ and $\oobw_k.$  This is the content
of the following derivation:
\begin{align}
  \y_{k+1} &= \x_{k+1}+(\sqrt{\kappa}-1)(\x_{k+1}-\x_k) &&\mbox{(by \eref{eq:ag.ydef}
and \eref{eq:taudef})} \notag \\
&=\sqrt{\kappa}\x_{k+1}-(\sqrt{\kappa}-1)\x_k \notag \\
&=\sqrt{\kappa}\w_k -(\sqrt{\kappa}-1)\x_k- \sqrt{\kappa}\nabla f(\w_k)/L 
&& \mbox{(by \eref{eq:wupd})} \notag \\
&=\sqrt{\kappa}\w_k -(\sqrt{\kappa}-1)\x_k- \kappa^{-1/2}\nabla f(\w_k)/\ell 
&& \mbox{(by \eref{eq:kappadef})} \notag \\
&=\kappa^{-1/2}\w_k+(1-\kappa^{-1/2})((\sqrt{\kappa}+1)\w_k
-\sqrt{\kappa}\x_k) \notag \\
&\hphantom{=}\quad\mbox{}
- \kappa^{-1/2}\nabla f(\w_k)/\ell  \notag \\
&=\kappa^{-1/2}\w_k+(1-\kappa^{-1/2})\y_k- \kappa^{-1/2}\nabla f(\w_k)/\ell  &&\mbox{(by \eref{eq:wxy})} \notag\\
& = \kappa^{-1/2}\oobw_k + (1-\kappa^{-1/2})\y_k. &&\mbox{(by \eref{eq:ag.oobw})}
\label{eq:ag.lambdadef}
\end{align}

We take
$\x$, $\y$ appearing in  Lemma~\ref{lem:ballintersect1} to be
$\oobw_k$, $\y_k$ respectively.
Next we need to define $\delta,\rho,\sigma$ to be used
in Lemma.
In the case of $\rho$ and 
we copy the definitions used in the analysis of IA:
\begin{align}
\Vert\oobw_k-\x^*\Vert^2 &\le 
\frac{\Vert\nabla f(\w_k)\Vert^2}{\ell^2} - \frac{2(f(\w_k)-f(\x^*))}
{\ell} && \mbox{(by \eref{eq:rdef0})} \notag\\
&\equiv \rho^2. \label{eq:ag.rhodef}
\end{align}

For $\sigma$, we use the induction hypothesis:
\begin{align}
\Vert\y_k-\x^*\Vert^2&\le \tilde\sigma_k^2 - \frac{2 (f(\x_k)-f(\x^*))}{\ell} \notag\\
&\equiv \sigma^2. \label{eq:ag.sigmadef}
\end{align}

In the case of $\delta$, we have:
\begin{align}
\Vert \oobw_k-\y_k\Vert^2 &= 
\Vert \w_k-\nabla f(\w_k)/\ell - 
(\sqrt{\kappa}+1)\w_k+\sqrt{\kappa}\x_k)\Vert^2 &&
\mbox{(by \eref{eq:ag.oobw} and \eref{eq:wxy})} \notag \\
& = \Vert \sqrt{\kappa}(\x_k-\w_k)-\nabla f(\w_k)/\ell\Vert^2 \notag\\
&=\kappa\Vert\w_k-\x_k\Vert^2 + \frac{2\sqrt{\kappa}(\w_k-\x_k)^T\nabla f(\w_k)}
{\ell} \notag \\
&\hphantom{=}\quad\mbox{}
+ \frac{\Vert\nabla f(\w_k)\Vert^2}{\ell^2} \notag\\
&\ge 
\kappa\Vert\w_k-\x_k\Vert^2 + \frac{2\sqrt{\kappa}(f(\w_k)-f(\x_k))}{\ell} \notag\\
&\hphantom{=}\quad\mbox{}
+\sqrt{\kappa}\Vert\w_k-\x_k\Vert^2
+ \frac{\Vert\nabla f(\w_k)\Vert^2}{\ell^2}  && \mbox{(by strong convexity)} \notag\\
&=(\kappa + \sqrt{\kappa})
\Vert\w_k-\x_k\Vert^2 + \frac{2\sqrt{\kappa}(f(\w_k)-f(\x_k))}{\ell} \\
&\hphantom{=}\quad\mbox{}
+ \frac{\Vert\nabla f(\w_k)\Vert^2}{\ell^2} \notag\\
&\equiv \delta^2. \label{eq:ag.deltadef}
\end{align}

From \eref{eq:ag.lambdadef}, $\lambda = 1-\kappa^{-1/2}$
(and hence $\lambda(1-\lambda)=\kappa^{-1/2}-\kappa^{-1}$).
Finally, by Lemma~\ref{lem:ballintersect1},
\begin{align*}
\Vert\y_{k+1}-\x^*\Vert^2 &\le \kappa^{-1/2}\rho^2+(1-\kappa^{-1/2})\sigma^2
-(\kappa^{-1/2}-\kappa^{-1})\delta^2 \\
&=
\frac{\Vert\nabla f(\w_k)\Vert^2}{L\ell}-\frac{2(f(\w_k)-f(\x^*))}{\ell}
+(1-\kappa^{-1/2})\tilde\sigma_k^2-(\kappa^{1/2}-\kappa^{-1/2})\Vert\w_k-\x_k\Vert^2 \\
&=
\tilde\sigma_{k+1}^2-\frac{2(f(\x_{k+1})-f(\x^*))}{\ell},
\end{align*}
thus completing the induction.  The second line was obtained 
by substituting \eref{eq:ag.rhodef}, \eref{eq:ag.sigmadef}, \eref{eq:ag.deltadef}
in the first line
followed by cancellation of like terms.  The third line was obtained
from \eref{eq:ag.tsigma}.
\end{proof}

\section{Relationship between IA and AG}
\label{sec:AGanalysis2}

The relationship between the idealized algorithm and AG is weaker than
that between IA and either GD or CG because AG does not solve any
optimization subproblems and instead takes fixed stepsizes.
Thus, at best it is an approximation to IA.  Furthermore,
the computations of $\x_k$ and $\y_k$ are more closely tied together,
making it unclear whether any kind of induction hypothesis can be
applied to either in isolation. For these reasons, we propose the following
theorem characterizing the AG--IA relationship.

\begin{theorem}
  Suppose one is given points $\x_{k-1},\y_{k-1}\in\R^n$,
  a nonzero vector $\g\in\R^n$, $n\ge 3$, and
  scalars $\ell,L$ such that $L>\ell>0$.
  For the remainder of this discussion, define $\kappa=L/\ell$
  and
  \begin{equation}
    \w_{k-1}=\frac{\sqrt{\kappa}}{\sqrt{\kappa}+1} \x_{k-1}+\frac{1}{\sqrt{\kappa+1}}\y_{k-1},
    \label{eq;wxy3}
  \end{equation}
  as in \eref{eq:wxy0}.
  Assume further (these assumption will be explained later) that
  \begin{equation}
    \Vert \x_{k-1}-\w_{k-1}\Vert \le \Vert\g\Vert/L,
    \label{eq:assum5}
  \end{equation}
  and
  \begin{equation}
    (\x_{k-1}-\w_{k-1})^T\g=0.
    \label{eq:assum6}
  \end{equation}
  Let
  \begin{align*}
    \mathcal{F} = \{f:\R^n\rightarrow\R: & \mbox{$\nabla f$ is $L$-Lipschitz}, \\
    & \mbox{$f$ is strongly convex with modulus $\ell$}, \\
    &\nabla f(\w_{k-1})=\g\}.
  \end{align*}
  Define
  \begin{equation}
  q(\x,\y)=\sup_{f\in\mathcal{F}}\frac{(2/\ell)(f(\x)-\min\{f\})+\Vert\y-\argmin\{f\}\Vert^2}
  {(2/\ell)(f(\x_{k-1})-\min\{f\})+\Vert\y_{k-1}-\argmin\{f\}\Vert^2}.
  \label{eq:qdef}
  \end{equation}
  Here $\min\{f\}$ and $\argmin\{f\}$ are short-hand for
  $\min\{f(\x):\x\in\R^n\}$ and $\argmin\{f(\x):\x\in\R^n\}$ respectively.
  Define $\mathcal{M}_k=\w_{k-1}+\Span\{\g,\x_{k-1}-\y_{k-1}\}$.
  Let $\x_k,\y_k$ be the vectors computed by the AG
  algorithm for this data (which lie in $\mathcal{M}_k$).   Then
  \begin{equation}
    q(\x_k,\y_k)\le 1-\kappa^{-1/2},
    \label{eq:agbound1}
  \end{equation}
  and
  \begin{equation}
    \min\{q(\x,\y):\x,\y\in\mathcal{M}_k\}\ge 1-\kappa^{-1/2}-O(\kappa^{-1}).
    \label{eq:agbound2}
  \end{equation}
\end{theorem}
Note that \eref{eq:agbound1} and \eref{eq:agbound2} imply that
the choice of new iterate $(\x_k,\y_k)$ made by AG is optimal up to a lower
order remainder term given the partial information used by AG.  In
contrast, IA is optimal (separately) for both terms in the numerator
of \eref{eq:qdef} for the specific $f$ and with no remainder term.

\begin{proof}
The proof of \eref{eq:agbound1} appears in \cite{KarimiVavasis2016}
(see (21), (22) and (43) in that paper, which use different notation
for AG).  Therefore, this proof establishes only \eref{eq:agbound2},
which involves constructing
a certain $f\in\mathcal{F}$ to attain this bound.
It suffices to prove the result for the $n=3$
case, since we can extend $f:\R^3\rightarrow\R$ to higher dimensions by adding
terms $\ell x_4^2+\cdots+\ell x_n^2$ and appending 0's to $\x_{k-1}$,
$\y_{k-1}$ and $\g$.

The theorem is invariant under rigid motions of $\R^3$, so we
can place $\w_{k-1}$ at an arbitrary point.  In addition,
we can rotate the two
vectors $\g$ and $\x_{k-1}-\w_{k-1}$ to arbitrary positions as long
as their lengths and their orthogonality are preserved. Starting with
$\g$, let $g_0$ denote $\Vert\g\Vert$. 
Rotate $\g$ to
$$\g=\gamma\left(
\begin{array}{c}
  -L\kappa^{-3/4}(=-\ell \kappa^{1/4}) \\
  0 \\
  \frac{\ell}{1-\kappa^{-1}} \\
\end{array}
\right),$$
where $\gamma$ is chosen to preserve the length of $\g$, in other words,
$$\gamma=(g_0/\ell) \cdot (\kappa^{1/2}+(1-\kappa^{-1})^{-2})^{-1/2}.$$
Translate $\w_{k-1}$ as follows:
$$\w_{k-1}=
\gamma\left(
\begin{array}{c}
  -\kappa^{-3/4} \\
  0 \\
  \frac{1}{1-\kappa^{-1}}
\end{array}
\right).$$
Finally, we rotate $\x_{k-1}-\w_{k-1}$ as follows:
$$\x_{k-1}-\w_{k-1}=
\gamma\left(
\begin{array}{c}
  \frac{\kappa^{-5/4}v^2(1-\kappa^{-1})}{(1-\kappa^{-1})^2+\kappa^{-1/2}} \\
  v\kappa^{-3/4}\\
  v^2\kappa^{-1}-\frac{\kappa^{-3/2}v^2}{(1-\kappa^{-1})^2+\kappa^{-1/2}}
\end{array}
\right)
$$
where $v$ is a scalar parameter that controls the length
of $\x_{k-1}-\w_{k-1}$.  In more detail, observe that
$\Vert\x_{k-1}-\w_{k-1}\Vert=\gamma\kappa^{-3/4}|v|+O(\kappa^{-1})$,
while $\Vert\g\Vert/L=\gamma\kappa^{-3/4}+O(\kappa^{-1})$.  Therefore,
to assure \eref{eq:assum5}, we restrict $|v|\le 1$.
It is also straightforward
to check that \eref{eq:assum6} is satisfied.
With these two definitions in hand, we can now write:
$$\x_{k-1}=\w_{k-1}+(\x_{k-1}-\w_{k-1})=
\gamma\left(
\begin{array}{c}
  -\kappa^{-3/4}+\frac{\kappa^{-5/4}v^2(1-\kappa^{-1})}{(1-\kappa^{-1})^2+\kappa^{-1/2}}  \\
  v\kappa^{-3/4}\\
  \frac{1}{1-\kappa^{-1}} +v^2\kappa^{-1}-\frac{\kappa^{-3/2}v^2}{(1-\kappa^{-1})^2+\kappa^{-1/2}}
\end{array}
\right),
$$
and
$$\y_{k-1} = \w_{k-1}-\sqrt{\kappa}(\x_{k-1}-\w_{k-1})=
\gamma\left(
\begin{array}{c}
  -\kappa^{-3/4}-\frac{\kappa^{-3/4}v^2(1-\kappa^{-1})}{(1-\kappa^{-1})^2+\kappa^{-1/2}}  \\
  -v\kappa^{-1/4}\\
  \frac{1}{1-\kappa^{-1}} -v^2\kappa^{-1/2}+\frac{\kappa^{-1}v^2}{(1-\kappa^{-1})^2+\kappa^{-1/2}}
\end{array}
\right).$$

Next, we define:
$$f(\x)=\frac{Lx_1^2+\sqrt{L\ell}x_2^2+\ell x_3^2}{2}.$$
It is straightforward to verify that $\nabla f$ is $L$-Lipschitz,
that $f$ is $\ell$-strongly convex, and that
$\nabla f(\w_{k-1})=\g$; thus $f\in\mathcal{F}$.  Also,
it is obvious that $\argmin\{f\}=\bz$ and $\min\{f\}=0$.

We evaluate the two terms in the denominator of \eref{eq:qdef}:
\begin{align*}
  f(\x_{k-1}) - \min\{f\} &= f(\x_{k-1}) \\
  &=\frac{\gamma^2}{2}
\left(L(\kappa^{-3/4}+O(\kappa^{-5/4}))^2+
\sqrt{L\ell}(v\kappa^{-3/4})^2+
\ell(1+O(\kappa^{-1}))^2\right) \\
&= \frac{\gamma^2\ell}{2}(1+\kappa^{-1/2}+O(\kappa^{-1})).
\end{align*}
Also,
\begin{align*}
  \Vert\y_{k-1}-\argmin\{f\}\Vert^2 &= \Vert\y_{k-1}\Vert^2 \\
  &=\gamma^2\left(O(\kappa^{-3/4})^2+
  v^2\kappa^{-1/2}+
  (1-v^2\kappa^{-1/2}+O(\kappa^{-1}))^2\right) \\
  &=\gamma^2(1-v^2\kappa^{-1/2}+O(\kappa^{-1})).
\end{align*}
Thus, the denominator of \eref{eq:qdef} simplifies to
$\gamma^2(2+(1-v^2)\kappa^{-1/2}+O(\kappa^{-1}))$.

Next, we need to solve two constrained
quadratic optimization problems to obtain a lower bound on
the numerator of \eref{eq:qdef}.  The constraint is
$\x\in\mathcal{M}_{k}$ for the first and
$\y\in\mathcal{M}_{k}$ for the second.
Imposing the constraint is simpler if it is written
as an inhomogenous linear equation;
one checks that
\begin{align*}
  \mathcal{M}_{k}&=\w_{k-1}+\Span\{\x_{k-1}-\w_{k-1},\g\} \\
  &=\left\{\x\in\R^3:\frac{\kappa^{-1/4}x_1}{1-\kappa^{-1}}-v\kappa^{-1/4}x_2+x_3=\gamma\right\},
\end{align*}
by substituting $\w_{k-1}$, $\x_{k-1}-\w_{k-1}$ and $\g$ into the
left-hand side of given equation and confirming that the values are $\gamma, 0, 0$
respectively.
It is also straightforward to check using a Lagrange-multiplier
argument that for any positive $A,B,C$ and any $(a,b,c)\ne (0,0,0)$,
$$\min\{Ax_1^2+Bx_2^2+Cx_3^2:ax_1+bx_2+cx_3=\gamma\}=\frac{\gamma^2}{a^2/A+b^2/B+c^2/C}.$$
Using this result for the first term of the numerator of \eref{eq:qdef}
yields that for any $\x\in\mathcal{M}_k$,
\begin{align*}
f(\x)-\min\{f\}&\ge \frac{1}{2}\cdot\frac{\gamma^2}{
  \frac{\kappa^{-1/2}}{L(1-\kappa^{-1})^2} +
  \frac{v^2\kappa^{-1/2}}{\sqrt{Ll}} +
  \frac{1}{\ell}} \\
&= \frac{\gamma^2\ell}{2}(1+O(\kappa^{-1})).
\end{align*}
As for the second term of the numerator, for any $\y\in\mathcal{M}_k$,
\begin{align*}
  \Vert\y-\argmin\{f\}\Vert^2 &\ge 
  \frac{\gamma^2}{
    \frac{\kappa^{-1/2}}{(1-\kappa^{-1})^2} +
    v^2\kappa^{-1/2} +
    1} \\
  &= \gamma^2(1-(1+v^2)\kappa^{-1/2}+O(\kappa^{-1})).
\end{align*}
Therefore, a lower bound on the numerator of \eref{eq:qdef}
is $\gamma^2(2-(1+v^2)\kappa^{-1/2}+O(\kappa^{-1}))$.
Finally, for any $\x,\y\in \mathcal{M}_k$,
\begin{align*}
q(\x,\y)&\ge \frac{\gamma^2(2-(1+v^2)\kappa^{-1/2}+O(\kappa^{-1}))} 
{\gamma^2(2+(1-v^2)\kappa^{-1/2}+O(\kappa^{-1}))}\\
&=1-\kappa^{-1/2}+O(\kappa^{-1}).
\end{align*}
\end{proof}

We now turn to the assumptions of the preceding theorem.
An assumption like \eref{eq:assum5} is necessary because,
in the situation that $\Vert\g\Vert\ll \Vert \x_{k-1}-\w_{k-1}\Vert$,
strong convexity (see \eref{eq:rdef0})
implies that the true minimizer $\x^*$ is
close to $\w_{k-1}-\g/\ell$, meaning that the update
to $\y_k$ implicit in AG is suboptimal. Thus,
an assumption along the lines of
\eref{eq:assum5}
is necessary to establish the optimality of AG.

Orthogonality assumption \eref{eq:assum6} appears to be unnecessary
and rather is a limitation of our construction, which uses
a quadratic function $f$.  Intuitively,
we need to construct a function
that varies more rapidly in one direction than another.  We
used a quadratic function, whose level curves have fixed orthogonal axes,
which creates a requirement of orthogonality in the two directions.
However, a more general convex function may have level curves
whose axes of elongation vary from one level curve to the next.

As in the concluding remarks
of Section~\ref{sec:GDanalysis2},
the proof of the optimality of the algorithm combined with
consideration of the assumptions uncovers situations when the algorithm
may be making suboptimal choices.  In the case of AG, this occurs
on iterations when $\Vert\nabla f(\w_{k-1})\Vert$ is unexpectedly small.

\section{A hybrid nonlinear conjugate gradient}
\label{sec:hybrid}

In this section, we propose a hybrid nonlinear conjugate gradient algorithm with
a convergence guarantee for smooth, strongly convex functions which is
related to an algorithm from the PhD thesis of the first author
\cite{karimi:thesis}.  Classical
nonlinear conjugate gradient (NCG) methods such as the Fletcher-Reeves and
Polak-Ribi\`ere methods are known to have poor worst-case performance
for this class of functions--worse even than steepest descent.  See
\cite{NemYud83} for more information.  The method developed in
this section guarantees $O(\log(1/\epsilon)\sqrt{L/\ell})$ convergence,
the best possible, and reduces to the optimal CG 
algorithm in the case of a quadratic function.

The algorithm proposed below uses classical nonlinear conjugate
gradient steps mixed with geometric descent steps.
The rationale for developing this algorithm is as follows.  Classical NCG, although
it has no global convergence bound even for strongly convex functions,
behaves well on ``nearly quadratic'' functions.  For typical objective
functions occurring in practice, nearly quadratic behavior is expected close
to the solution.  Therefore, a method that can switch between steps with
a guaranteed complexity and NCG steps has the possibility of outperforming
both methods.

A summary of the algorithm is as follows.  At the beginning of
iteration $k$, the algorithm has a quadruple
$(\x_{k-1},\y_{k-1},\p_{k-1},\tilde\sigma_{k-1})$.  From this
quadruple, a step of nonlinear conjugate gradient can be
applied.  For the line search, the line-search function of $\alpha$,
namely, $f(\x_{k-1}+\alpha\p_{k-1})$, is approximated by a univariate
quadratic, whose quadratic coefficient is obtained by computing
$\p_{k-1}^T\nabla^2 f(\x_{k-1})\p_{k-1}$ using reverse-mode automatic
differentiation.  This approximation is exact in the case that $f$
itself is a quadratic function, in which case the hybrid algorithm
reproduces the steps of linear conjugate gradient.

The hybrid algorithm then computes $\x_k=\x_{k-1}+\alpha\p_{k-1}$ and
computes $\y_k$ 
as in the GD algorithm.  It checks whether $f$ has decreased and
whether $\tilde\sigma_{k}^2\le (1-\sqrt{\ell/L})\tilde\sigma_{k-1}^2$.
If so, the iteration is over, and the nonlinear CG step is accepted.
If not, then a GD step is taken instead.
The detailed specification of the algorithm is as follows.

\begin{align}
& \mbox{\bf Hybrid Nonlinear Conjugate Gradient} \notag\\
& \x_0:=\mbox{arbitrary};\quad\y_0:=\x_0;\quad\p_0:=\bz\notag \\
& \g_{-1}:=\bz;
\tilde\sigma_0 := \sqrt{2}\Vert\nabla f(\x_0)\Vert/\ell \notag\\
& \mbox{for }k = 1,2,\ldots, \notag\\
& \hphantom{\mbox{for }} \g_{k-1}:=\nabla f(\x_{k-1})\notag\\
& \hphantom{\mbox{for }} (\x_k^{\rm CG}, \p_k) :=
{\it CGSTEP}(\x_{k-1},\p_{k-1},\g_{k-2},\g_{k-1},k) \notag\\
& \hphantom{\mbox{for }} (\y_k,\tilde\xi^*_k) :=
{\it YCOMPUTE}(\x_{k-1},\g_{k-1},\y_{k-1},\tilde\sigma_{k-1}) \notag\\
& \hphantom{\mbox{for }} \hat\gamma_k^{\rm CG} := 2(f(\x_k^{\rm CG})-f(\x_{k-1}))/\ell 
\label{eq:hatgamma1}\\
& \hphantom{\mbox{for }} \tilde\sigma_k^{\rm CG} := 
\sqrt{(\tilde\xi^*_k)^2+
\hat\gamma_k^{\rm CG}}\notag \\
& \hphantom{\mbox{for }}\mbox{if } \hat\gamma_k^{\rm CG} \le 0 \mbox{ and }
(\tilde\sigma_k^{\rm CG})^2\le \left(1-\sqrt{\ell/L}\right)\tilde\sigma_{k-1}^2 
\notag\\
&\hphantom{\mbox{for if }} \x_{k}:=\x_{k-1}^{\rm CG} \notag\\
&\hphantom{\mbox{for if }} \tilde\sigma_{k}:=\tilde\sigma_k^{\rm CG} \notag\\
& \hphantom{\mbox{for }}\mbox{else}\notag \\
& \hphantom{\mbox{for if }} \obx_{k-1}:=\x_{k-1}-\g_{k-1}/L \notag\\
&\hphantom{\mbox{for if }} \x_k:=\argmin\{f(\x):\x\in\aff\{\obx_{k-1},\y_k\}\}
\label{eq:linesearch}\\
& \hphantom{\mbox{for if }} \hat\gamma_k := 2(f(\x_k)-f(\x_{k-1}))/\ell \label{eq:hatgamma2}\\
& \hphantom{\mbox{for if }} \tilde\sigma_k := 
\sqrt{(\tilde\xi^*_k)^2+
\hat\gamma_k}\notag \\
& \hphantom{\mbox{for if }} \p_k := \x_k-\x_{k-1} \label{eq:newp} \\
& \hphantom{\mbox{for }} \mbox{end} \notag \\
& \mbox{end} \notag 
\end{align}

\begin{align}
&\mbox{{\bf Function }
${\it CGSTEP}(\x_{k-1},\p_{k-1},\g_{k-2},\g_{k-1},k)$} \notag\\
&\mbox{if }k == 1 \notag \\
& \hphantom{\mbox{if }}\p_k := -\g_{k-1} \notag \\
& \mbox{else} \notag\\
& \hphantom{\mbox{if }} \z := \g_{k-1}-\g_{k-2} \notag \\
& \hphantom{\mbox{if }}\beta_{k} := 
\frac{1}{\z^T\p_{k-1}}\left(
\z - \frac{2\p_{k-1}\Vert\z\Vert^2}{\z^T\p_{k-1}}
\right)^T\g_{k-1} \label{eq:hzbeta} \\
& \hphantom{\mbox{if }}\p_k:= \beta_k\p_{k-1}-\g_{k-1} \notag\\
& \mbox{end} \notag\\
& \alpha_k := -\frac{\p_k^T\g_{k-1}}{\p_k^T\nabla^2 f(\x_{k-1})\p_k} \label{eq:qstep} \\
& \x_k := \x_{k-1}+\alpha_k\p_k \notag \\
& \mbox{return } (\x_k,\p_k) \notag
\end{align}

\begin{align}
&\mbox{{\bf Function }
${\it YCOMPUTE}(\x_{k-1},\g_{k-1},\y_{k-1},\tilde\sigma_{k-1})$} \notag\\
& \oobx_{k-1} := \x_{k-1}-\g_{k-1}/\ell\notag\\
& \tilde\rho_{k-1} := \Vert \g_{k-1}\Vert /\ell \notag \\
& \mbox{if } \tilde\sigma_{k-1}^2 \le 2\tilde\rho_{k-1}^2 \notag \\
& \hphantom{\mbox{if }} \delta_{k-1} := \Vert\y_{k-1}-\oobx_{k-1}\Vert \label{eq:compdelta}\\
& \hphantom{\mbox{if }}\mbox{if }
\delta_{k-1} > \tilde\rho_{k-1} \mbox{ and } 
\tilde\rho_{k-1}>|\tilde\rho_{k-1}^2-\tilde\sigma_{k-1}^2|^{1/2}
\label{eq:if2} \\
& \hphantom{\mbox{if if }}
\lambda_k := 
\frac{\delta_{k-1}^2+\tilde\rho_{k-1}^2-\tilde\sigma_{k-1}^2}{2\delta_{k-1}^2} \notag \\
& \hphantom{\mbox{if if }}\tilde\xi^*_k :=
\frac{1}{2}\sqrt{2\tilde\rho_{k-1}^2+2\tilde\sigma_{k-1}^2-\delta_{k-1}^2-
\frac{(\tilde\rho_{k-1}^2-\tilde\sigma_{k-1}^2)^2}{\delta_{k-1}^2}} \notag \\
& \hphantom{\mbox{if }}\mbox{else} \notag\\
& \hphantom{\mbox{if if }}\lambda_k := 1 \notag \\
& \hphantom{\mbox{if if }}\tilde\xi_k^*:= \tilde\sigma_{k-1} \notag \\
& \hphantom{\mbox{if }} \mbox{end} \notag \\
& \mbox{else} \notag \\
& \hphantom{\mbox{if }} \lambda_k := 0 \notag \\
& \hphantom{\mbox{if }} \tilde\xi_k^* := \tilde\rho_{k-1} \notag \\
& \mbox{end} \notag \\
& \y_k := (1-\lambda_k)\oobx_{k-1} + \lambda_k\y_{k-1} \notag \\
& \mbox{return } (\y_k,\tilde\xi_k^*) \notag
\end{align}

Some remarks on this procedure are as follows.  The variable
$\hat\gamma_k$ in \eref{eq:hatgamma1}
and \eref{eq:hatgamma2} stands for $\gamma_k-\gamma_{k-1}$, where 
$\gamma_k$ is defined as in \eref{eq:gammakdef}.
The line-search implicit in \eref{eq:linesearch} is carried
out with a univariate Newton method.  Because we have not made
sufficient assumptions about $f$ to guarantee convergence of 
Newton's method, the Newton method is safeguarded with a bisection method.
The univariate second derivative of $f$ needed for the Newton method
can be computed using reverse-mode automatic differentiation in time
proportional to evaluate $f(\x)$ (refer to \cite{NocedalWright}).  This
univariate second derivative is also needed in \eref{eq:qstep}.

The formula for $\beta_k$ in \eref{eq:hzbeta} is from the
CG-Descent algorithm of Hager and Zhang \cite{HagerZhang}.  As mentioned
earlier, the formula for $\alpha_k$ appearing in \eref{eq:qstep} is
based on a univariate quadratic Taylor expansion of the line-search function
at $\x_k$ in the direction $\p_k$.  This formula is exact if $f$ itself
is quadratic, but in all other cases it is speculative.  However, if it
yields a poor answer, the overall algorithm is still robust because 
when the CG step gives a poor answer, the GD algorithm serves as a backup.

The main theorem about this method, which follows from the material
presented so far, is as follows.

\begin{theorem}
Assuming exact line-search in \eref{eq:linesearch}, 
the Hybrid NCG algorithm produces a sequence of iterates
$(\x_k,\y_k,\sigma_k)$ satisfying \eref{eq:sigma-ub} and
\eref{eq:sigma-decr}.  Furthermore, if $f(\x)$ is a quadratic function,
then Hybrid NCG produces the same sequence of iterates as linear
conjugate gradient.
\end{theorem}

We now turn to three important numerical issues with this method.
The first issue to note is that function {\it YCOMPUTE} has an ``if''
statement \eref{eq:if2} not present in \eref{eq:gd.rpdef}--\eref{eq:sigmadef2}
when the GD algorithm was described.  
In the case of the ``exact'' GD algorithm, the condition of the
if-statement is guaranteed to hold as established in Section~\ref{sec:GDanalysis}.
However, because the line-search is only approximate, \eref{eq:gradorth} does
not hold exactly, and therefore the condition of the \eref{eq:if2} may occasionally
fail.  In this case, we safeguard its failure by defining $\y_k:=\y_{k-1}$
and $\sigma_k:=\sigma_{k-1}$ (so that $\tilde\sigma_k$ is updated only due
to the decrease in the objective), i.e., we keep the same containing sphere for
the optimizer as on the previous step.

The second numerical issue concerns the computation of $\delta_{k-1}$ in 
\eref{eq:compdelta}.  This formula is prone to roundoff error as the algorithm
converges because $\y_{k-1}$ and $\oobx_{k-1}$ both tend to $\x^*$ in the limit.
In our implementation,
we addressed this issue by maintaining a separate program variable storing
the vector $\y_k-\x_k$.  This vector is updated using a recurrent formula
that is straightforward to derive; the recurrence updates $\y_k-\x_k$ using
vectors that also tend to $\bz$.  Given an accurate representation
of $\y_k-\x_k$ it is clear that 
\eref{eq:compdelta} can be computed without significant roundoff issues.
A similar issue and similar workaround is applied to \eref{eq:newp}.

The third numerical issue concerns the subtractions in
\eref{eq:hatgamma1}
and \eref{eq:hatgamma2}, which are also prone to roundoff error
as $\x_k$  converges.  These errors could upset
the computation of $\tilde\sigma_k$.
Our implementation addressed this using ``computational divided
differences''; see, e.g., \cite{divdiff}.

\section{Computational Results}
\label{sec:comp}

We implemented four methods: Geometric Descent (GD), Accelerated Gradient (AG),
Hybrid Nonlinear Conjugate Gradient (HyNCG, described in the preceding section), 
and NCG.  NCG stands for nonlinear conjugate gradient using the Hager-Zhang
formula for $\beta_k$ given by \eref{eq:hzbeta}.  (The entirety of their
NCG method is called ``CG-Descent''; however, we did not implement other aspects
of CG-Descent such as the line search.)
The line search used by
GD, HyNCG and NCG is based on Newton's method and
is safeguarded with a bisection.  The techniques to address numerical
issues described in the
preceding section were applied in GD, HyNCG and NCG. (The line
search and the numerical techniques are not needed for AG).

We applied these four methods to two problem classes:
approximate BPDN and hinge-loss halfspace classification.

BPDN (basis pursuit denoising) refers to the 
unconstrained convex optimization
problem:
$$\min \Vert A\x-\b\Vert^2 +\lambda\Vert\x\Vert_1$$
in which $\lambda > 0$ and $A\in\R^{m\times n}$ 
has fewer rows than columns, so that the problem
is neither strongly convex nor smooth.  However, the following
approximation (called APBDN)
is both smooth and strongly convex on any bounded domain:
$$\min \Vert A\x-\b\Vert^2 +\lambda\sum_{i=1}^n\sqrt{x_i^2+\delta}$$
where $\delta>0$ is a fixed scalar.  It is easy to see that as
$\delta\rightarrow 0$, the original problem is recovered.  As
$\delta\rightarrow 0$, $\ell\rightarrow 0$ and $L\rightarrow\infty$,
where $\ell,L$ are the moduli of strong, smooth convexity.

In our tests of ABPDN 
we took $A$ to be a subset of $\sqrt{n}$ rows
of the discrete-cosine transform matrix of size $n\times n$, where
$n$ is an even power of 2.  (This matrix and its transpose, although
dense, can be applied in $O(n\log n)$ operations.)  The subset of
rows was selected to be those numbered by the first $m=\sqrt{n}$ prime
integers
in order to get reproducible
pseudorandomness in the choices.  Similarly, in 
order to obtain a pseudorandom $\b$, we selected $\b\in\R^m$ 
according to the formula $b_i=\sin(i^2)$.  The value of
$\lambda$ was fixed at $10^{-3}$ in all tests; the convergence
criterion was $\Vert\nabla f(\x_k)\Vert\le 10^{-8}$.
 Finally,
we varied $\delta=10^{-2},10^{-3},10^{-4}$ and we tried both
$n=65536$ and $n=262144$.

The second test-case is the hinge-loss (HL) function for half-space
identification taken from \cite{bubeck}, which is as follows:
$f(\x)=H(\b\circ(A\x))+\lambda\Vert\x\Vert^2/2$, where
$A$ is a given $m\times n$
matrix,
$\b$ is a given $m$-vector of $\pm 1$,  `$\circ$' denotes Hadamard product (i.e., the entrywise
product of two vectors), $\lambda>0$ is a regularization parameter, and
$H(\v)=\sum_{i=1}^mh(v_i)$ where 
$$h(v)=\left\{
\begin{array}{ll}
0.5-v, & v\le 0, \\
(1-v)^2/2, & v\in[0,1], \\
0, & v\ge 1.
\end{array}
\right.$$
Minimizing $f(\cdot)$ corresponds to finding a hyperplane determined
by $\x$ of the form $U=\{\z\in\R^n:\x^T\z=0\}$ such that rows $i$ of $A$,
$i=1,\ldots,m$, for which
$b_i=1$ lie on one side of $U$ (i.e., $A(i,:)\x>0$)
while rows $i$ of $A$ for which
$b_i=-1$ lie on the opposite side (i.e., $A(i,:)\x<0$).  
The objective function penalizes misclassified
points as well as penalizing a large value of $\x$.

This function is smooth and strongly convex.  As $\lambda\rightarrow 0$, the
strong convexity parameter $\ell$ vanishes.

Unlike \cite{bubeck}, who test GD applied
to this function on data sets available
on the web, we have tested the four algorithms on synthetic data for the
purpose of better control over experimental conditions.  
In our tests, $m=200000$, $n=447$ (so that $n\approx\sqrt{m}$), 
$\lambda=3\cdot 10^{-1}, 3\cdot 10^{-2}, 3\cdot 10^{-3}$.
For each $i$, $i=1,\ldots, m$, $b(i)=\pm 1$ chosen at random with
probability $0.5$.  If $b(i)=1$, then $A(i,:)=[1,\ldots,1]/\sqrt{n}+\w_i^T$,
where $\w_i$ is a noise vector chosen as a spherical Gaussian with 
covariance matrix ${\rm diag}(\sigma^2,\ldots,\sigma^2)$, where $\sigma=0.4$.
If $b(i)=-1$, then $A(i,:)=-[1,\ldots,1]/\sqrt{n}+\w_i^T$.
For these tests, the convergence test was $\Vert \nabla f(\x_k)\Vert \le 10^{-6}$.

The results of all tests are shown in Table~\ref{tab:compresults}.  
For NCG and GD, the numbers in this table are the number of inner iterations (line
search steps), which is the dominant cost in these algorithms.
In the case of HyNCG, we have reported the sum of the number of CG steps
(which do not require a line-search) plus the number of inner line-search
iterations.  For AG we have
reported the number of outer iterations.
The notation DNC indicates that the algorithm did achieve the
requisite tolerance after
$10^5$ outer iterations.

\begin{table}
\begin{center}
\begin{tabular}{|l|rrrr|}
\hline
 &  GD & AG & NCG & HyNCG \\
\hline
ABPDN, $n=65536$, $\delta=10^{-2}$ & 58,510 & DNC & 12,345 & 757 \\
ABPDN, $n=65536$, $\delta=10^{-3}$ & 314,367 & DNC & DNC & 9,510 \\
ABPDN, $n=65536$, $\delta=10^{-4}$ & 585,362& DNC & DNC &  28,395 \\
ABPDN, $n=262144$, $\delta=10^{-2}$ & 7,734 & 34758 & 488 & 123 \\
ABPDN, $n=262144$, $\delta=10^{-3}$ & 782,223 & DNC & DNC & 17,195 \\
ABPDN, $n=262144$, $\delta=10^{-4}$ & DNC & DNC & DNC & 40,328 \\
HL, $m=200000$, $\lambda =0.3$ & 154 & 13,170 &  112 & 37 \\
HL, $m=200000$, $\lambda =0.03$ & 151 & 29,218 & 110 & 37 \\ 
HL, $m=200000$, $\lambda =0.003$ & 151 & 58,793 & 113 & 44 \\ 
\hline
\end{tabular}
\end{center}
\caption{Number of iterations (see text for details) of four algorithms
on nine synthetic test cases.}
\label{tab:compresults}
\end{table}
One sees from the table that HyNCG was superior in every test case, sometimes by
a wide margin.  An unexpected feature of the table, for which we currently do not have
an explanation, is that in the case of the HL suite of problems, the 
number of iterations was nearly invariant with respect to variation in $\lambda$,
except for AG, whose running time grows steadily with decreasing $\lambda$.

To conclude this section, we also consider two hybrid algorithm
that do not use the potential.  They are as follows:
compute a step of both GD and CG, and then select the step that 
decreases either $\Vert \nabla f(\x_k)\Vert$ (denoted HyNCG/gr) or
$f(\x_k)$ (denoted HyNCG/f) by the greatest amount.  The results of
this experiment are presented in Table~\ref{tab:compresults2}.

\begin{table}
\begin{center}
\begin{tabular}{|l|rrr|}
\hline
 &  HyNCG & HyNCG/gr & HyNCG/f \\
\hline
ABPDN, $n=65536$, $\delta=10^{-2}$ & 757  & 1,346 & 36,708 \\
ABPDN, $n=65536$, $\delta=10^{-3}$ & 9,510& 24,509  & 252,050  \\
ABPDN, $n=65536$, $\delta=10^{-4}$ & 28,395 & 67,354  & 105,947 \\
ABPDN, $n=262144$, $\delta=10^{-2}$ & 123 & 184 & 394 \\
ABPDN, $n=262144$, $\delta=10^{-3}$ & 17,195 & 37,435 & 652,435 \\
ABPDN, $n=262144$, $\delta=10^{-4}$ & 40,328 & 86,445 & 213,226 \\
HL, $m=200000$, $\lambda =0.3$ & 37 & 57 &  45 \\
HL, $m=200000$, $\lambda =0.03$ & 37 & 57 &  45 \\ 
HL, $m=200000$, $\lambda =0.003$ & 44 & 64 & 61 \\ 
\hline
\end{tabular}
\end{center}
\caption{Number of iterations (see text for details) of three
different hybrids on nine synthetic test cases.}
\label{tab:compresults2}
\end{table}

The table shows that the hybrid
based on the potential outperforms the other two methods, often
by a factor of 2 and sometimes
by a large factor.
The reason for this follows from
the discussion in Section~\ref{sec:computable}.
Although the norm of the gradient can be used (and in fact, {\em was} used
for all tests in this section) as a termination criterion, it is not
helpful for measuring progress step by step.
The two methods HyNCG/gr and HyNCG/f must carry out two evaluations
per iteration to decide which step is preferable.
In contrast, the hybrid HYNCG based on the
potential can select the CG step without trying an alternative
provided the potential shows sufficient decrease.  

\section{Conclusions}
We have demonstrated that a single computable potential bounds the
convergence of three algorithms, conjugate gradient, accelerated gradient
and geometric descent.  We have also pointed out other connections between
the algorithms, namely, their relationship to an idealized algorithm and
their relationship to the Bubeck-Lee-Singh lemma.
The existence of this potential enables the formulation
of a hybrid method for convex optimization that duplicates the steps of conjugate
gradient in the case of conjugate gradient but nonetheless achieves the
optimal complexity for general smooth, strongly convex problems.  
Directions for future work include the following.
\begin{itemize}
\item
The hybrid algorithm requires prior knowledge of $\ell,L$; it would be interesting
to develop an algorithm with the same guarantees that does not need prior knowledge
of them.  Note that linear conjugate gradient does not need any such prior knowledge
of the coefficient matrix $A$.
\item
It would be interesting to establish a theoretical result about the improved
performance of the hybrid algorithm in the case of ``nearly quadratic'' functions.
\item
Although accelerated gradient has been extended well beyond the realm
of unconstrained smooth, strongly convex functions, none of the other
algorithms has been.  It would be interesting to extend the conjugate
gradient ideas outside this space.  
Also interesting is the extension to constrained or composite convex
minimization.  See, for example, \cite{KarimiVavasis}.
\end{itemize}
\bibliographystyle{plain}
\bibliography{optimization}
\end{document}